\documentclass[12pt]{amsart}
\usepackage[utf8]{inputenc}
\usepackage{a4wide}
\usepackage{lmodern}

\usepackage{hyperref}
\usepackage{dsfont}
\usepackage{stmaryrd}
\usepackage{todonotes}
\usepackage{cite}
\usepackage[capitalize]{cleveref}

\usepackage{amsmath,amsfonts,amssymb,amsthm}
\usepackage{bbm}
\usepackage{pgf,tikz,pgfplots}
\pgfplotsset{compat=1.14}
\usepackage{mathrsfs}
\usetikzlibrary{arrows}

\allowdisplaybreaks

\usepackage{amsmath}
\usepackage{amssymb}
\usepackage{amsthm}
\usepackage{amsfonts}

\newcommand{\N}{\mathbb{N}}

\newcommand{\C}{\mathbb{C}}

\theoremstyle{plain}
\newtheorem{thm}{Theorem}[section]
\newtheorem{lem}[thm]{Lemma}
\newtheorem{prop}[thm]{Proposition}
\newtheorem{cor}[thm]{Corollary}
\newtheorem{dfn}[thm]{Definition}
\newtheorem*{cond}{Condition (C)}
\newtheorem*{conv}{Convention}

\theoremstyle{definition}

\theoremstyle{remark}
\newtheorem*{rem}{Remark}

\newtheorem*{nota}{Notation}

\usepackage{xspace}
\def\Vs{\mathcal{V}\xspace}
\def\Vsr{$\mathcal{V}$\xspace}

\title[Graph classes with few $P_4$'s: Universality and Brownian graphon limits]{Graph classes with few $P_4$'s: Universality and Brownian graphon limits}
 \author{Théo Lenoir}

\begin{document}
\maketitle

\begin{abstract}
We consider large uniform labeled random graphs in different classes
with few induced $P_4$ ($P_4$ is the graph consisting of a single line of
$4$ vertices) which generalize the case of cographs. Our main result
is the convergence to a Brownian limit object in the space of
graphons. As a by-product we obtain new asymptotic enumerative results for all these graph classes.
We also obtain typical density results for a wide variety of induced subgraphs. These asymptotics hold at a smaller scale than what is observable through the graphon convergence.

Our proofs rely on tree encoding of graphs. We then use mainly
combinatorial arguments, including the symbolic method and singularity
analysis.
\end{abstract}

\section{Introduction}

\subsection{Motivation}
Random graphs are one of the most studied objects in probability theory and in combinatorics. A natural question is to investigate the scaling limits of a uniformly chosen graph in a given family (an important example for this paper are the cographs).

Cographs have been studied since the seventies by various authors, especially for their algorithmic properties: recognizing cographs can be solved in linear time \cite{algo1, algo2, algo3}, and many hard problems can be solved in polynomial time for cographs. Several equivalent definitions exists of the class of cographs exists, here are two important ones:

\begin{itemize}
    \item A graph is a cograph if and only if it has no induced $P_4$ (a line of $4$ vertices).
    \item The class of cograph is the smallest class containing every graph reduced to a single vertex, and stable by union and by join\footnote{the join of two graphs $(G,H)$ is the graph obtained by adding an edge between every pair of vertices $(g,h)\in G\times H$}.
\end{itemize}

Simultaneously in \cite{bassino2021random} and \cite{bs}, the authors exhibit a Brownian limit object for a uniform cograph, called the Brownian cographon, which can be explicitly constructed from the Brownian excursion and a parameter $p\in[0,1]$. 

The convergence holds in distribution in the sense of \emph{graphons}. Introduced in \cite{graphon}, graphons are a well-established topic in graph theory but their probabilistic counterparts are more recent. Graphon convergence can be seen as the convergence of the renormalized adjacency matrix for the so-called cut metric (a good reference on graphon theory is \cite{lovasz}).

To go further than the case of cographs, we may investigate more complicated classes with, in some specific sense, few $P_4$'s. A natural question is to study classes of graphs to which some algorithmic properties of cographs extend. Several classes characterized by properties of their induced $P_4$'s have thus been considered in the graph theory literature. The classes we will focus on here are the following: $P_4$-reducible graphs \cite{reduc, reductible}, $P_4$-sparse graphs \cite{sparse, jamison} $P_4$-lite graphs \cite{lite}, $P_4$-extendible graphs \cite{extendible} and $P_4$-tidy graphs \cite{tidy} which can all be seen as classes defined by some constraints on the induced $P_4$'s. All these classes will be defined precisely in \Cref{sec:cla}. The inclusion relations between these classes are sketched in \Cref{fig1}.

\begin{figure}[htbp]
\begin{center}
\centering
\includegraphics[scale=0.8]{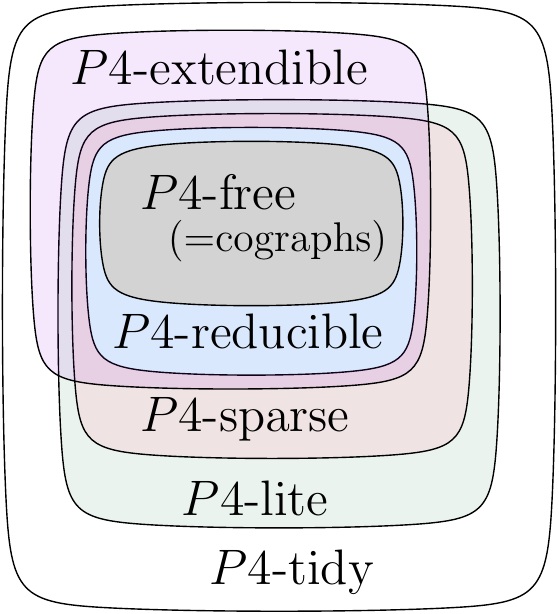}
\caption{Inclusion relations between the different classes of graphs}\label{fig1}
\end{center}
\end{figure}

To our knowledge, these different classes have not been studied from a probabilistic point of view. The main aim of this paper is to prove a result of universality of the Brownian cographon: for every class previously mentioned, a random graph will converge towards the Brownian cographon of parameter $\frac{1}{2}$ (the rigorous construction is given by \cite[Definition 10]{bassino2021random}). An intermediate result is the asymptotic enumeration of each of these classes, which was unknown up to now.

\subsection{Main results}

For a finite graph $G$, let $W_G$ be the embedding of the finite graph $G$ in the set of graphon (the formal construction will be recalled in \Cref{graphe-graphon}). Our main result is:

\begin{thm}\label{cgbintro}
Let $\mathbf{G}^{(n)}$ be a graph of size $n$ taken uniformly at random in one of the following families: $P_4$-sparse, $P_4$-tidy, $P_4$-lite, $P_4$-extendible or $P_4$-reducible. The following convergence in distribution holds in the sense of graphons:

$$W_{\mathbf{G}^{(n)}}\stackrel{n\to \infty}{\longrightarrow} \mathbf{W}^{\frac{1}{2}}$$

\noindent where $\mathbf{W}^{\frac{1}{2}}$ is the Brownian cographon of parameter $\frac{1}{2}$.
\end{thm}

Graphon convergence is equivalent to the joint convergence of subgraphs density. Diaconis and Janson extended this criterion in \cite{janson} to random graphs: the convergence of a family $(\mathbf{H}^{(n)})_{n\geq 1}$ of random graphs is characterized by the convergence in distribution of $\frac{\mathrm{Occ}_{\mathbf{H}^{(n)}}(H)}{n^k}$ for every positive integer $k$ and for every finite graph $H$ of size $k$, where $\mathrm{Occ}_G(H)$ is the number of induced subgraphs of $G$ isomorphic to $H$. All the necessary material on graphon will be recalled at the beginning of \Cref{sec6}.

\Cref{fig2} shows an example of the adjacency matrix of a random $P_4$-extensible graph of size $200$. This picture gives an idea of what a realization of the Brownian cographon could look like.

\begin{figure}[htbp]
\begin{center}
\centering
\includegraphics[scale=1]{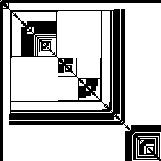}
\caption{The adjacency matrix of a random $P_4$-extensible graph of size $200$, simulation by Mickaël Maazoun}\label{fig2}
\end{center}
\end{figure}

In the course of proving \Cref{cgbintro}, we get an equivalent of the number of graphs in the different classes.

\begin{thm}\label{enumintro}
The number of labeled $P_4$-sparse, $P_4$-tidy, $P_4$-lite, $P_4$-extendible, $P_4$-reducible or the number of $P_4$-free graphs of size $n$ is asymptotically equivalent to 
$$C \frac{n!}{R^nn^{\frac{3}{2}}},$$
for some $R,C>0$, depending on the class.
\end{thm}

We can compute with arbitrary precision the numerical values of $R$ and $C$ (see \cref{sec4.2}). All the numerical values of $R$ and $C$ vary according to each class which confirms that all these classes are significantly different.

\Cref{cgbintro} provides a precise estimation of $\mathrm{Occ}_{H}(\mathbf{G}^{(n)})$ for every cograph $H$. But for every graph $H$ which is not a cograph, the only information given by the convergence in the sense of graphon is that the number of induced $H$ in $\mathbf{G}^{(n)}$ is typically $o(n^{|H|})$. Quite unexpectedly, thanks to the tools developed to prove \Cref{cgbintro}, we are able to estimate the expected number of induced subgraphs isomorphic to a specific class of graphs $H$ in $\mathbf{G}^{(n)}$: the graphs that are called "prime" for the modular decomposition (see \Cref{prime}). 

\begin{thm}\label{P4intro}
Let $\mathbf{G}^{(n)}$ be a graph of size $n$ taken uniformly at random in one of the following families: $P_4$-sparse, $P_4$-tidy, $P_4$-lite, $P_4$-extendible or $P_4$-reducible.
Let $H$ be a prime graph, denote by $\mathrm{Occ}_{H}(\mathbf{G}^{(n)})$ the number of labeled subgraphs of $\mathbf{G}^{(n)}$ isomorphic to $H$.

Then there exists $K_H\geq0$ such that:
$$\mathbb{E}[\mathrm{Occ}_{H}(\mathbf{G}^{(n)})]\sim\begin{cases}
     &K_H n^{\frac{3}{2}}\qquad \text{if $H$ verifies condition $(A)$}\\
    &K_H n\qquad \text{otherwise}
\end{cases}$$
where $(A)$ is defined p.\pageref{thmprime} and constant $K_H$ is given in \Cref{thmprime}.
\end{thm}
This results follows from \Cref{thmprime} which is stated in a more general setting.
The condition $(A)$ depends on the class of graphs, checking if $H$ verifies condition $(A)$ and if $K_H$ is positive is quite straightforward.

To make things more concrete, let us apply \Cref{P4intro} to the example of $H=P_4$. We can check that for each class $P_4$ does not verify condition $(A)$ and that $K_{P_4}>0$. Thus a uniform random graph contains in average a linear number of induced $P_4$, while \Cref{cgbintro} only implies that this number is $o(n^4)$. The different numerical values of $K_{P_4}$ are explicitly computed p.\pageref{cora}, and happen to take different values for each class. For each class, the graph called \emph{bull} (see \cref{bull}) verifies condition $(A)$ and that $K_{\mathrm{bull}}>0$. Thus a uniform random graph contains in average a number of induced bulls growing as $n^{3/2}$, while \Cref{cgbintro} only implies that this number is $o(n^5)$. However, for non prime graphs $H$, the behavior of the expected value of induced subgraphs of $\mathbf{G}^{(n)}$ isomorphic to $H$ is not well-understood, which leads to interesting open questions.

\subsection{Proof strategy}

The proof is essentially combinatorial and is based on modular decomposition, which allows to encode a graph with a decorated tree. Modular decomposition is a standard tool in graph theory (it was introduced in the $60$'s by Gallai \cite{Gallai}) but to our knowledge it has been very little used in the context of random graphs. In this paper we introduce an enriched modular decomposition which enables us to obtain exact enumerations for a large family of graph classes. The five classes mentioned before fit in this framework.
We exploit those enumerative results with tools from analytic combinatorics to get asymptotic estimates in order to prove \Cref{enumintro}.

The more technical part of the proof is, for every finite graph $H$, to estimate the number of induced subgraphs of $\mathbf{G}^{(n)}$ isomorphic to $H$. The enriched modular decomposition allows us to count the number of graphs with a specific induced subgraph $H$. Again asymptotics are derived with tools from combinatorics to prove \Cref{cgbintro} and \Cref{P4intro}.

\subsection{Outline of the paper}

\begin{itemize}
    \item In \Cref{sec2} we define the encoding of graphs with trees, the modular decomposition and the enriched modular decomposition which will be used throughout the different proofs.
    \item  \Cref{sec:cla} presents the necessary material on the different classes of graphs studied: results are already widely known, most of them are quoted from the litterature and reformulated to suit our enriched modular decomposition.
    \item \Cref{sec4,sec5} are about calculating generating series related to our graph classes: in \Cref{sec4} we prove \Cref{enumintro} and \Cref{sec5} deals with the generating series of graphs with a given induced subgraph.
    \item \Cref{sec6} presents the necessary material on graphons, and the proofs of \Cref{cgbintro} and \Cref{P4intro}.
\end{itemize}

\section{Modular decomposition of graphs: old and new}\label{sec2}
\subsection{Labeled graphs}
In the following all the graphs considered are simple and finite.  Each time a graph $G$ is defined, we denote by $V$ its set of vertices and $E$ its set of edges. Whenever there is an ambiguity, we denote by $V_G$ (resp.~$E_G$) the set of vertices (resp.~edges) of $G$.

\begin{dfn}

We say that $G=(V,E)$ is a \emph{weakly-labeled graph} if every element of $V$ has a distinct label in $\N$ and that $G=(V,E)$ is a \emph{labeled graph} if every element of $V$ has a distinct label in $\{1,\dots, |V|\}$. 

The \emph{size of a graph} $G$, denoted by $|G|$, is its number of vertices. 

The \emph{minimum of a graph} $G$, denoted $\mathrm{min}(G)$, is the minimal label of its vertices.

\end{dfn}

In the following, every graph will be labeled, otherwise we will mention explicitly that the graph is weakly-labeled.

\begin{rem}
We do not identify a vertex with its label. A vertex of label $i$ will be denoted $v_i$. The label of a vertex $v$ will be denoted $\ell(v)$.
\end{rem}

\begin{dfn}\label{red}
For any weakly-labeled object (graph or tree) of size $n$, we call \emph{reduction} the operation that reduces its labels to the set $\{1,\dots,n\}$ while preserving the relative order of the labels.
\end{dfn}

For example if $G$ has labels $2,4,12,63$ then the reduced version of $G$ is a copy of $G$ in which  $2, 4, 12, 63$ are respectively replaced by $1,2,3,4$.

\subsection{Encoding graphs with trees}

\begin{dfn}\label{2.3}
Let $G$ be a graph of size $n$ and $H_1,\dots, H_{n}$ be weakly-labeled graphs such that no label is given to two distinct vertices of $\bigcup_{i=1}^{n} H_i$. The graph \emph{$G[H_1,\dots, H_{n}]=(V,E)$} is the graph whose set of vertices is $V=\bigcup_{i=1}^n V_{H_i}$ and such that:
\begin{itemize}
\item for every $i\in \{1,\dots,n\}$ and every pair $(v,v')\in V_{H_i}^2$, $\{v,v'\}\in E$ if and only if $\{v,v'\}\in E_{H_i}$;
\item for every $(i,j)\in \{1,\dots,n\}$ with $i\neq j$, and every pair $(v,v')\in V_{H_i}\times V_{H_j}$, $\{v,v'\}\in E$ if and only if $\{v_i,v_j\}\in E_G$.
\end{itemize}
\end{dfn}

\begin{nota}

In \Cref{2.3} we will use the shortcut $\oplus$ for the complete graph of size $n$. Thus $\oplus[H_1,\dots ,H_n]$ is the graph obtained from copies of $H_1,\dots,H_n$ in which for every $i\neq j$ every vertex of $H_i$ is connected to every vertex of $H_j$. This graph is called the \emph{join} of $H_1,\dots,H_n$ 

We use the shortcut $\ominus$ for the empty graph of size $n$. Thus $\ominus[H_1,\dots ,H_n]$ is the graph given by the disjoint union of $H_1,\dots,H_n$ This graph is called the \emph{union} of $H_1,\dots,H_n$.

\end{nota}

This construction allows us to transform non-plane labeled trees with internal nodes decorated with graphs, $\oplus$ and $\ominus$ into graphs.

\begin{dfn}\label{d3}
Let $\mathcal{T}_0$ be the set of rooted non-plane trees whose leaves have distinct labels in $\N$ and whose internal nodes carry decorations satisfying the following constraints: 
\begin{itemize}
    \item internal nodes are decorated with $\oplus$, $\ominus$ or a graph;
    \item If a node is decorated with some graph $G$ then $|G|\geq 2$ and this node has $|G|$ children. If a node is
decorated with $\oplus$ or $\ominus$ then it has at least 2 children.
   
\end{itemize}

A tree $t\in \mathcal{T}_0$ is called a \emph{substitution tree} if the labels of its leaves are in $\{1,\dots, |t|\}$.

\end{dfn}

We call \emph{linear} the internal nodes decorated with $\oplus$ or $\ominus$ and \emph{non-linear} the other ones.

\begin{nota}
For a non-plane rooted tree $t$, and an internal node $v$ of $t$, let $t_v$ be the multiset of trees attached to $v$ and let $t[v]$ be the non-plane tree rooted at $v$ containing only the descendants of $v$ in $t$. 
\end{nota}

\begin{conv}

We only consider non-plane trees. However it is sometimes convenient to order the subtrees of a given node. The convention is that for some $v$ in a tree $t$ the trees of $t_v$ are ordered according to their minimal leaf labels.

\end{conv}

\begin{dfn}\label{d2}
Let $t$ be an element of $\mathcal{T}_0$, the weakly-labeled graph \emph{$\mathrm{Graph}(t)$} is inductively defined as follows: 
\begin{itemize}
    \item if $t$ is reduced to a single leaf labeled $j$, $\mathrm{Graph}(t)$ is the graph reduced to a single vertex labeled $j$;
    \item otherwise, the root $r$ of $t$ is decorated with a graph $H$, and $$\mathrm{Graph}(t)=H[\mathrm{Graph}(t_1),\dots, \mathrm{Graph}(t_{|H|})]$$
    where $t_i$ is the $i$-th tree of $t_r$.
\end{itemize}
\end{dfn}

\begin{figure}[htbp]
\begin{center}
\includegraphics[scale=0.6]{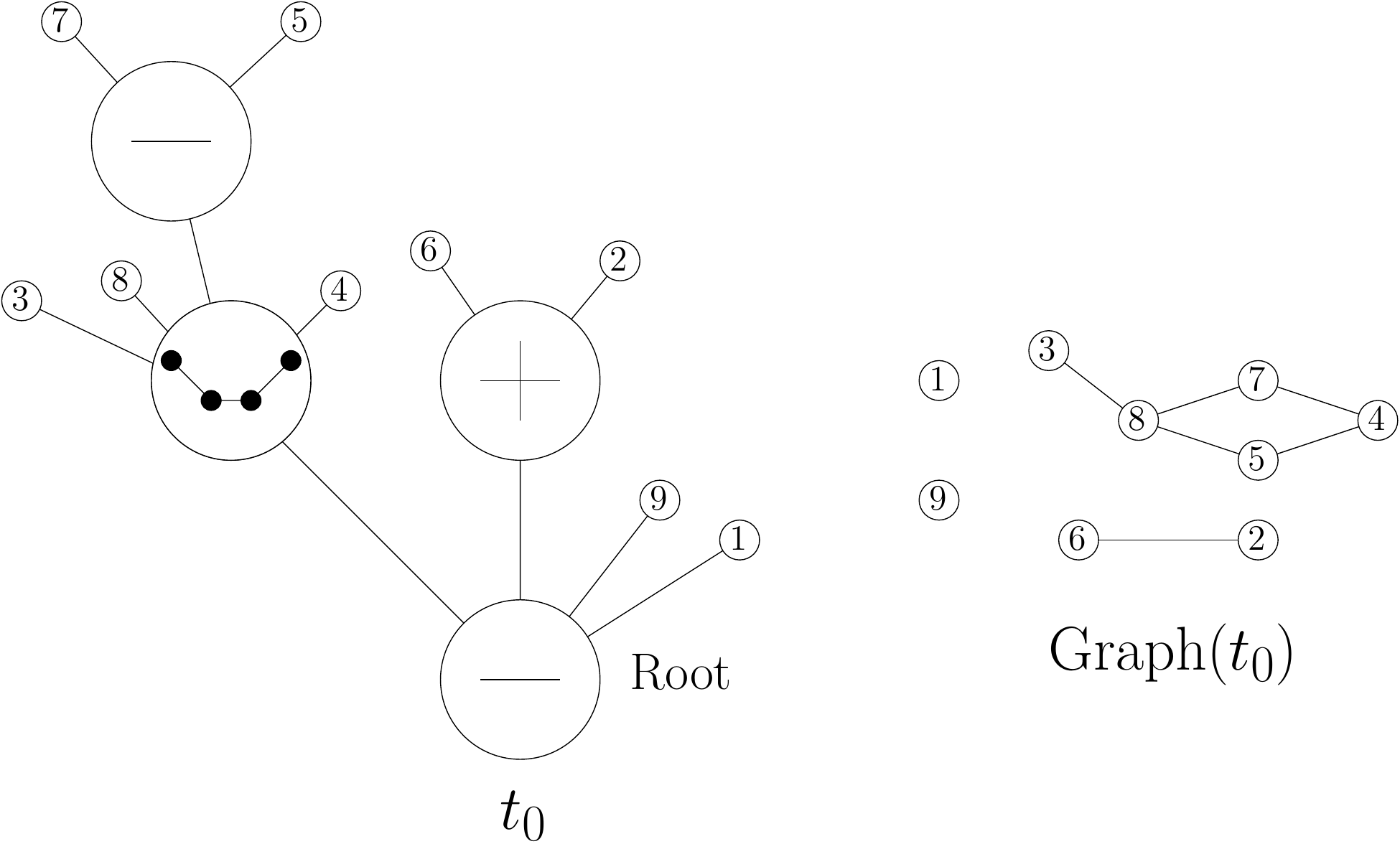}
\caption{A substitution tree $t_0$ and the corresponding graph $\mathrm{Graph}(t_0)$}
\end{center}
\end{figure}

Note that if $t$ is a substitution tree then $\mathrm{Graph}(t)$ is a labeled graph.

The following simple Lemma is essential to the study of the enriched decomposition
of graphs introduced in \Cref{enriched}.
\begin{lem}\label{l1}
Let $t$ be a substitution tree such that the decoration of the root of $t$ (resp.~its complementary) is connected. Then $\mathrm{Graph}(t)$ (resp.~its complementary) is connected. 

\end{lem}

\begin{proof}
Since both cases are similar, we only deal with the case of a connected decoration.
Let $r$ be the root of $t$, $H$ its decoration and $k$ the size of $H$. Let $w_1,\dots, w_k$ be vertices of $\mathrm{Graph}(t)$ such that for each $i\in \{1,\dots, k\}$ there is a leaf labeled $\ell(w_i)$ in the $i$-th tree of $t_r$. Since the unlabeled graph induced by $\{w_i \mid 1\leq i \leq k\}$ is isomorphic to $H$, it is connected. Let $C$ be the connected component of $\mathrm{Graph}(t)$ containing all $w_i$'s. Note that for every vertex $v$ of $\mathrm{Graph}(t)$, there exists $p\in \{1,\dots k\}$ such that the leaf labeled $\ell(v)$ belongs to the $p$-th tree of $t_r$. Since $H$ is connected and of size at least $2$, there exists $q\neq p$ such that the vertices of label $q$ and $p$ are connected by an edge in $H$. Thus $v$ and $w_q$ are connected by an edge in $\mathrm{Graph}(t)$, which means that $v\in C$. This implies that $C=\mathrm{Graph}(t)$, thus $\mathrm{Graph}(t)$ is connected.\end{proof}

\subsection{Modular decomposition}

In this short section we gather the main definitions and properties of modular decomposition. The historical reference is \cite{Gallai}, the interested reader may also look at \cite{graphclasses} or \cite{algo}.

The next definitions and theorems allows to get a unique recursive decomposition of any graph in the sense of \Cref{d2}, the modular decomposition, and to encode it by a tree.

\begin{dfn}
Let $G$ be a graph (labeled or not). A \emph{module} $M$ of $G$ is a subset of $V$ such that for every $(x,y)\in M^2$, and every $z\in V\backslash M$, $\{x,z\}\in E$ if and only if $\{y,z\}\in E$.
\end{dfn}

\begin{rem}
Note that $\emptyset, V$ and $\{v\}$ for $v\in V$ are always modules of $G$. Those sets are called the trivial modules of $G$.
\end{rem}

\begin{dfn}\label{prime}
A graph $G$ is \emph{prime} if it has at least $3$ vertices and its only modules are the trivial ones.
\end{dfn}

\begin{dfn}
A graph is called \emph{$\ominus$-indecomposable} (resp.~\emph{$\oplus$-indecomposable}) if it cannot be written as $\ominus[G_1,\dots,G_k]$ (resp.~$\oplus[G_1,\dots,G_k]$) for some $k\geq 2$ and weakly-labeled graphs $G_1,\dots, G_k$.
\end{dfn}

Note that a graph is $\ominus$-indecomposable if and only if it is connected, and $\oplus$-indecomposable if and only if its complementary is connected.

\begin{thm}[Modular decomposition, \cite{Gallai}]\label{thm1}
Let $G$ be a graph with at least $2$ vertices, there exists a unique partition $\mathcal{M}=\{M_1,\dots, M_k\}$ for some $k\geq 2$ (where the $M_i$'s are ordered by their smallest element), where each $M_i$ is a module of $G$ and such that either
\begin{itemize}
    \item $G=\oplus[M_1,\dots, M_k]$ and the $(M_i)_{1\leq i \leq k}$ are $\oplus$-indecomposable;
    \item $G=\ominus[M_1,\dots, M_k]$ and the $(M_i)_{1\leq i \leq k}$ are $\ominus$-indecomposable;
    \item there exists a unique prime graph $P$ such that $G=P[M_1,\dots, M_k]$.
\end{itemize}

\end{thm}

This decomposition can be used to encode graphs by specific trees to get a one-to-one correspondence.

\begin{dfn}
Let $t$ be a substitution tree. We say that $t$ is a \emph{canonical tree} if its internal nodes are either $\oplus$, $\ominus$ or prime graphs, and if there is no child of a node decorated with $\oplus$ (resp.~$\ominus$) which is decorated with $\oplus$ (resp.~$\ominus$).

\end{dfn}

To a graph $G$ we associate a canonical tree by recursively applying the decomposition of
\Cref{thm1} to the modules $(M_i)_{1\leq i \leq k}$, until they are of size $1$. First of all, at each step, we order the different modules increasingly according to their minimal vertex labels. Doing so, a labeled graph $G$ can be
encoded by a canonical tree. The internal nodes are decorated with the different graphs that are encountered along the recursive decomposition process ($\oplus$ if $G=\oplus[M_1,\dots, M_k]$, $\ominus$ if $G=\ominus[M_1,\dots, M_k]$, $P$ if $G=P[M_1,\dots, M_k]$).\\
At the end, every module of size $1$ is converted into a leaf labeled by the label of the vertex.

This construction provides a one-to-one correspondence
between labeled graphs and canonical trees that maps the size of a graph to the size of the corresponding tree.

\begin{prop}
Let $G$ be a graph, and $t$ its canonical tree, then $t$ is the only canonical tree such that $\mathrm{Graph}(t)=G$.
\end{prop}

\begin{rem}
It is crucial to consider canonical trees as non-plane: otherwise, since prime graphs can have several labelings, there would be several canonical trees associated with the same graph.
\end{rem}

\subsection{Enriched modular decomposition}\label{enriched}

Unfortunately the modular decomposition alone does not provide usable decompositions for the graph classes that we consider. The aim of this section is to solve this issue: we will state and prove \Cref{2.18} which provides in a very general setting a one-to-one encoding of graphs with substitution trees with constraints.
In \Cref{sec:cla} we will show that $P_4$-reducible graphs, $P_4$-sparse graphs, $P_4$-lite graphs, $P_4$-extendible graphs, $P_4$-tidy graphs fit in the settings of \Cref{2.18}.

\begin{dfn}
We say that $G$ is a \emph{graph with blossoms} if there exists $k\in\{0,\dots,|V|\}$ such that exactly $k$ vertices of $G$ are labeled $*$, and the others ones have a distinct label in $\{1,\dots,|V|-k\}$. 

The vertices labeled $*$ are called the \emph{blossoms} of $G$. Let $B_G$ the set of vertices that are blossoms of $G$ and $N(G):=|V|-|B_G|$ the number of vertices that are not a blossom of $G$.

\end{dfn}

\begin{rem}
In the above definition, we allow $k=0$, then the definition reduces to the one of a labeled graph.
\end{rem}

\begin{dfn}
Let $G$ be a graph with blossoms and $\pi$ be a permutation of $\{1,\dots,N(G)\}$. The \emph{$\pi$-relabeling} of $G$ is the graph $G'$ such that: 
\begin{itemize}
    \item $V_{G'}=V_G$ and $B_{G'}=B_G$;
    \item for every vertex $v$ in $V_{G'}\backslash B_{G'}$, we replace the label of the leaf $v$ by $\pi(\ell(v))$. 
    
\end{itemize}

We write $G\sim G'$ if there exists a permutation $\pi$ of $\{1,\dots, N(G)\}$ such that $G$ is isomorphic to the $\pi$-relabeling of $G'$.

\end{dfn}
Note that $\sim$ is an equivalence relation.

\begin{dfn}
Let $G$ be a graph with blossoms, a permutation $\pi$ of $\{1,\dots, N(G)\}$ is an \emph{automorphism} of $G$ if the $\pi$-relabeling of $G$ is $G$.

\end{dfn}

\begin{dfn}\label{defblo}
A module of a graph with blossoms is called \emph{flowerless} if it does not contain any blossom.\\
Let $G$ be a graph with blossoms and $M$ a non-empty flowerless module of $G$. We define \emph{$\mathsf{blo}_{M}(G)$} to be the labeled graph obtained after the following transformations:
\begin{itemize}

    \item $M$ is replaced by a new vertex $v$, that is now labeled $*$;
    \item for every vertex $w\in G\backslash M$, $\{w,v\}$ is an edge if and only if $\{w,m\}$ is an edge of $G$ for every $m\in M$;
    \item the graph obtained is replaced by its reduction as defined in \Cref{red}.

\end{itemize}

If $G$ is a graph with one blossom and $M$ is a non-empty flowerless module of $G$, we define $\mathsf{blo}_{M,0}(G)$ (resp.~$\mathsf{blo}_{M,1}(G)$) to be the graph $\mathsf{blo}_{M}(G)$ where the label of the initial blossom of $G$ is replaced by $*_0$ (resp.~$*_1$) and the label of the new blossom is replaced by $*_1$ (resp.~$*_0$).
\end{dfn}

\begin{figure}[htbp]
\begin{center}
\centering
\includegraphics[scale=0.6]{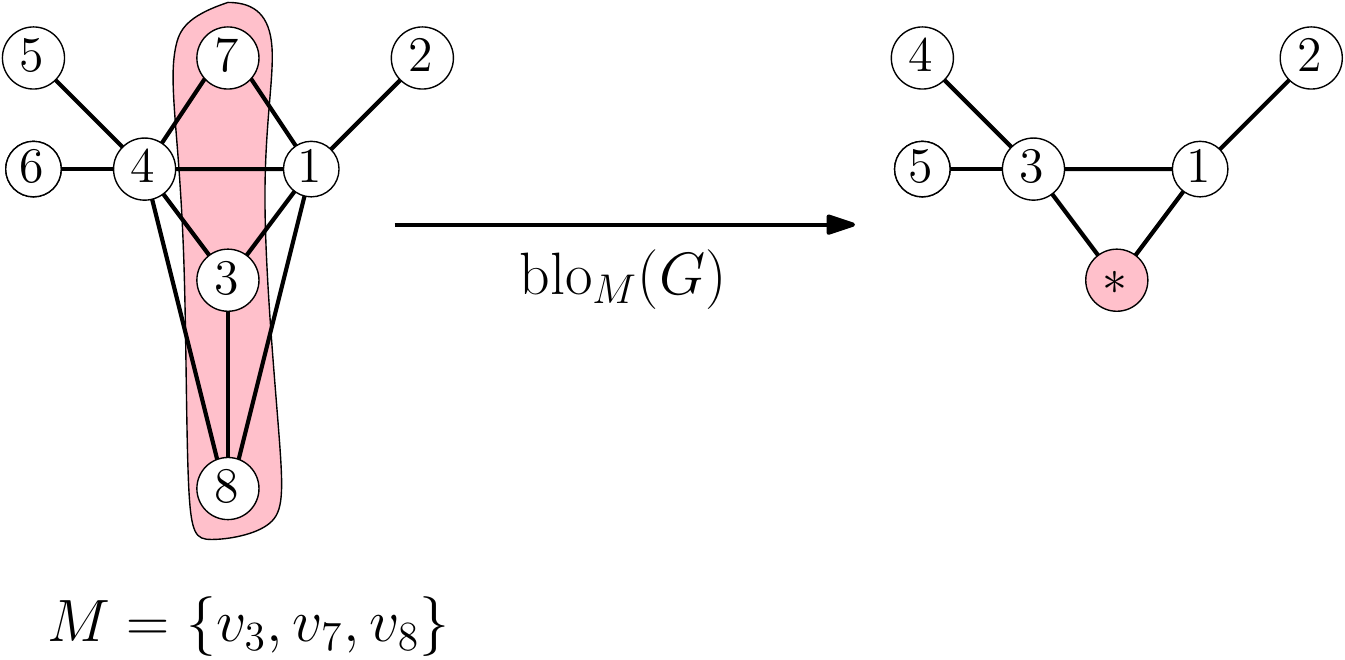}
\caption{Illustration of \Cref{defblo} Left: A graph $G$ in which we have highlighted the module $M=\{v_3,v_7,v_8\}$. Right: The corresponding $\mathsf{blo}_M(G)$.}
\end{center}
\end{figure}

\noindent{}In this paper, we only consider the construction $\mathsf{blo}_M(G)$ for graphs with $0$ or $1$ blossom.\\
We are now ready to precise the general framework of our study. One of the key ingredient is the following recursive definition of families of graphs.

\begin{dfn}\label{d1}
Let $\mathcal{P}$ be a set of graphs with no blossom and $\mathcal{P}^{\bullet}$  be a set of graphs with one blossom. A tree $t\in \mathcal{T}_{0}$ is called \emph{$(\mathcal{P},\mathcal{P}^{\bullet})$-consistent} if one of the following conditions holds:

\begin{enumerate}
\item[(D1)] The tree $t$ is a single leaf.

\item[(D2)] The root $r$ of $t$ is decorated with a graph $H\in \mathcal{P}$ and $t_r$ (the multiset of trees attached to $r$) is a union of leaves.
\item[(D3)] The root $r$ of $t$ is decorated with $\oplus$ (resp.~$\ominus$) and all the elements of $t_r$ are $(\mathcal{P},\mathcal{P}^{\bullet})$-consistent and their roots are not decorated with $\oplus$ (resp.~$\ominus$).

\item[(D4)]  The root $r$ of $t$ is decorated with a graph $H\notin \{\oplus,\ominus\}$ and there exists at least one index $i\in \{1,\dots, |H|\}$ such that the $i$-th tree of $t_r$ is $(\mathcal{P},\mathcal{P}^{\bullet})$-consistent, the remaining trees in $t_r$ are reduced to a single leaf and $\mathsf{blo}_{\{v_i\}}(H)\in \mathcal{P}^{\bullet}$.

\end{enumerate}
We define $\mathcal{T}_{\mathcal{P},\mathcal{P}^{\bullet}}$ to be the set of trees $t$ that are $(\mathcal{P},\mathcal{P}^{\bullet})$-consistent and such that each leaf has a distinct label in $\{1,\dots, |t|\}$.

\end{dfn}

\begin{figure}[htbp]
\begin{center}
\centering
\includegraphics[scale=0.6]{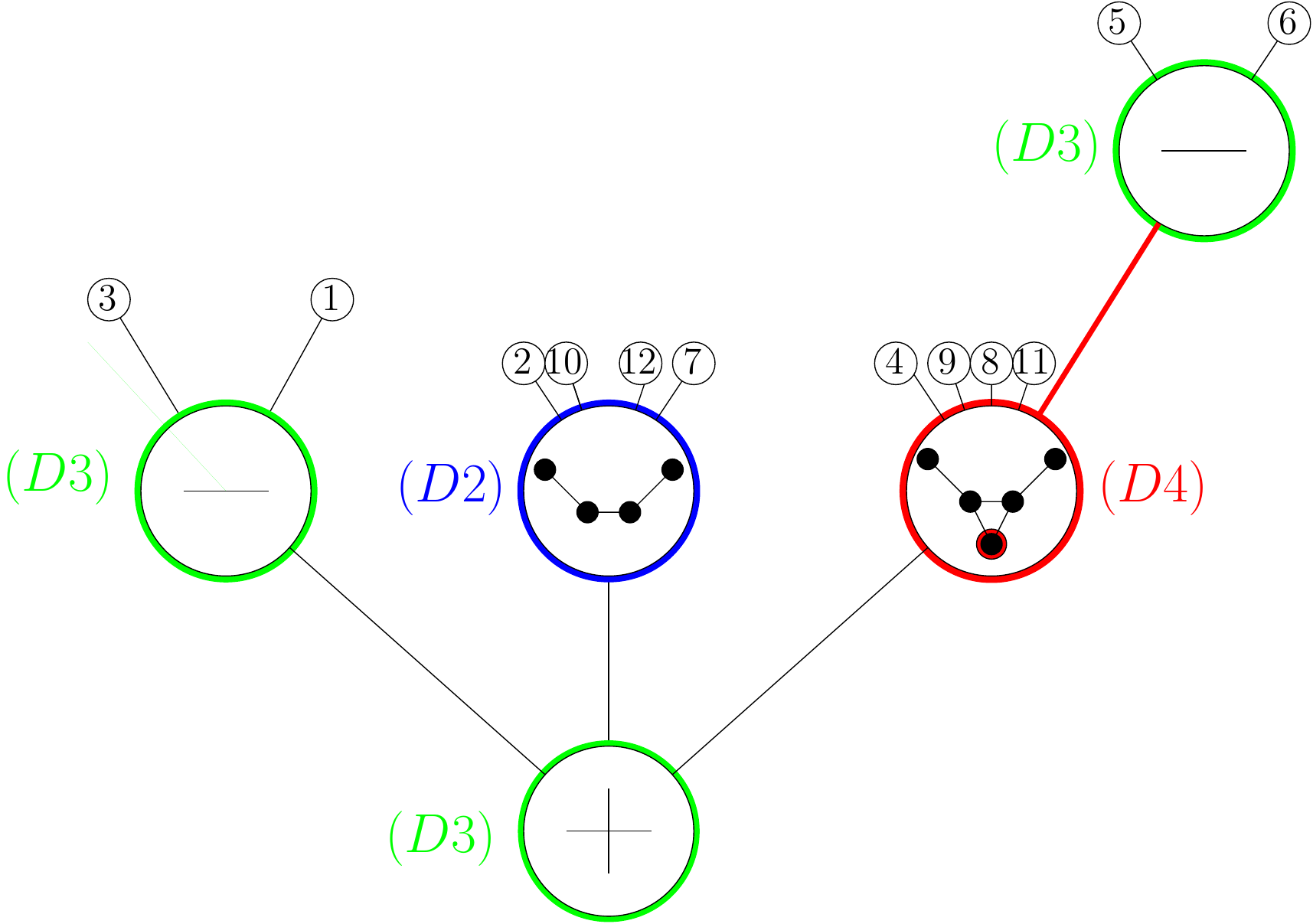}
\caption{An example of tree in some $\mathcal{T}_{\mathcal{P},\mathcal{P}^{\bullet}}$. The different colours illustrate the different cases of \Cref{d1}. The subtree with leaves $\{5,6\}$ on the top-right is attached to the vertex which is circled in red inside the vertex of case $(D4)$. This corresponds to the $i$-th subtree of case $(D4)$}
\end{center}
\end{figure}

A graph $G$ is called $(\mathcal{P},\mathcal{P}^{\bullet})$-consistent if there exists a $(\mathcal{P},\mathcal{P}^{\bullet})$-consistent tree $t$ such that $G=\mathrm{Graph}(t)$. We let $\mathcal{G}_{\mathcal{P},\mathcal{P}^{\bullet}}$ be the set of $\mathrm{Graph}(t)$ for  $t\in \mathcal{T}_{\mathcal{P},\mathcal{P}^{\bullet}}$.

The map $ t\mapsto \mathrm{Graph}(t)$ from $\mathcal{T}_{\mathcal{P},\mathcal{P}^{\bullet}}$ to $\mathcal{G}_{\mathcal{P},\mathcal{P}^{\bullet}}$ is surjective, but without conditions on $(\mathcal{P},\mathcal{P}^{\bullet})$ this map is not one-to-one.
To solve this issue, we introduce the following additional constraints on the set $\mathcal{P},\mathcal{P}^{\bullet}$:

\begin{cond}\label{cond1}\ 

\begin{enumerate}

    \item[(C1)] $\mathcal{P}$ and $\mathcal{P}^{\bullet}$ do not contain a graph of size $1$.
    \item[(C2)] For every $F\in \mathcal{P}$ and every module $M$ of $F$, either  $\mathsf{blo}_M(F)\not\in \mathcal{P}^{\bullet}$ or the subgraph of $F$ induced by $M$ is not $(\mathcal{P},\mathcal{P}^{\bullet})$-consistent.
    \item[(C3)] For every $F$ and $F'$ in $\mathcal{P}^{\bullet}$, and every flowerless modules $M$ and $M'$ of respectively $F$ and $F'$ 
    one of the following conditions is verified:

    \begin{itemize}
        \item $\mathsf{blo}_{M,0}(F)\neq \mathsf{blo}_{M',1}(F')$
        \item  The subgraph of $F$ induced by $M$  is not $(\mathcal{P},\mathcal{P}^{\bullet})$-consistent.
        \item The subgraph of $F'$ induced by $M'$  is not $(\mathcal{P},\mathcal{P}^{\bullet})$-consistent.
    \end{itemize}
    \item[(C4)] Every element of $\mathcal{P}$ and $\mathcal{P}^{\bullet}$ is $\oplus$-indecomposable and $\ominus$-indecomposable.
    \item [(C5)] For every $G\in \mathcal{P}^{\bullet}$, the only modules of $G$ containing the blossom are $\{*\}$ and $G$.

\end{enumerate}

We say that $(\mathcal{P},\mathcal{P}^{\bullet})$ verifies condition $(C)$ if $(C1)-(C5)$ hold.

\end{cond}

\begin{rem}
The last two constraints are not necessary to ensure that the map is bijective. However, giving necessary and sufficient conditions to have unicity that can be checked easily is quite complicated.

Note that if condition $(C)$ is satisfied for a pair of sets $(\mathcal{P},\mathcal{P}^{\bullet})$ and $\mathcal{Q}\subset \mathcal{P}$ and $\mathcal{Q}^{\bullet}\subset \mathcal{P}^{\bullet}$, it is also verified by $(\mathcal{Q},\mathcal{Q}^{\bullet})$.

\end{rem}

\begin{prop}\label{2.18}
Let $\mathcal{P}$ be a set of graphs with no blossom and $\mathcal{P}^{\bullet}$ a set of graphs with one blossom. Assume that $(\mathcal{P},\mathcal{P}^{\bullet})$ verifies condition $(C)$.
For any $G\in \mathcal{G}_{\mathcal{P},\mathcal{P}^{\bullet}}$, there exists a unique $t\in \mathcal{T}_{\mathcal{P},\mathcal{P}^{\bullet}}$ such that $G=\mathrm{Graph}(t)$. Moreover, for any element of $\mathcal{T}_{\mathcal{P},\mathcal{P}^{\bullet}}$ satisfying case $(D4)$ in \Cref{d1}, the index $i$ such that case $(D4)$ holds is unique.
\end{prop}

\begin{proof}

Existence is guaranted by definition of $\mathcal{G}_{\mathcal{P},\mathcal{P}^{\bullet}}$.

We proceed by contradiction to prove the uniqueness of $t$. Let $t$ be a smallest tree in $ \mathcal{T}_{\mathcal{P},\mathcal{P}^{\bullet}}$ such that there exists another $t'$ in $\mathcal{T}_{\mathcal{P},\mathcal{P}^{\bullet}}$ verifying $\mathrm{Graph}(t)=\mathrm{Graph}(t')$. Let $G=\mathrm{Graph}(t)$.

The graph $G$ cannot be reduced to a single vertex due to $(C1)$, otherwise $t$ and $t'$ would be a single leaf with label $1$. Thus we can assume that $t$ and $t'$ are not in case $(D1)$.

By \Cref{l1} and $(C4)$, $G$ is $\oplus$-indecomposable (resp.~$\ominus$-indecomposable) if and only if $t$ is not in case $(D3)$ with a root decorated with $\oplus$ (resp.~$\ominus$). Thus either $t$ and $t'$ are both in case $(D3)$ and their roots are both decorated $\oplus$ or $\ominus$, or they are both in case $(D2)$ or $(D4)$.

\noindent{\bf Case (i): $t,t'$ are both in case $(D3)$ and their are both decorated $\oplus$ or $\ominus$.}

Let $r$ and $r'$ be the roots of respectively $t$ and $t'$. Assume that both decorations are $\ominus$, the other case is similar. The elements of $t_r$ induce connected graphs by \Cref{l1} as their roots are either decorated with $\oplus$, or $\ominus$-indecomposable by $(C4)$. Since the roots of $t$ and $t'$ are decorated with $\ominus$, we have a one-to-one correspondence between trees of $t_r$ and connected components of $G$. The same is true for $t'_{r'}$. Assume that two trees corresponding to the same connected component of $G$ are different. Since their set of labels are the same (they correspond to the labels of the vertices in the connected component) after reduction, one would obtain two trees $t_1,t_2$ that are different, $(\mathcal{P},\mathcal{P}^{\bullet})$-consistent and such that $\mathrm{Graph}(t_1)=\mathrm{Graph}(t_2)$ since both are equal to the reduction of the corresponding connected component of $G$. This contradicts the minimality of $t$.
Therefore $t_r=t'_{r'}$ and $t=t'$.

\noindent{\bf Case (ii): $t,t'$ are both in case $(D2)$.}

The graph $G$ is just the decoration of both roots of $t$ and $t'$ so $t=t'$.

\noindent{\bf Case (iii): $t$ is in case $(D4)$, $t'$ is in case $(D2)$.}

Since $t'$ is in case $(D2)$, $G$ is just the decoration of the root of $t'$ thus $G\in\mathcal{P}$. Let $r$ be the root of $t$ and $H$ its decoration. Let $i$ be one of the elements of $\{1,\dots |V_H|\} $ such that $(D4)$ holds for $t,H$ and $i$. Let $M$ be the set of vertices of $G$ whose labels are labels of leaves that belong to the $i$-th tree of $t_{r}$: $M$ is a module of $G$. Then $\mathsf{blo}_M(G)$ is equal to $\mathsf{blo}_{\{v_i\}}(H)$ and thus belongs to $\mathcal{P}^{\bullet}$. Moreover the subgraph of $G$ induced by $M$ is $(\mathcal{P},\mathcal{P}^{\bullet})$-consistent as the $i$-th subtree of $t$ is also $(\mathcal{P},\mathcal{P}^{\bullet})$-consistent. This contradicts $(C2)$.

\noindent{\bf Case (iv): $t,t'$ are both in case $(D4)$.}

 Let $r$ and $r'$ be the roots of respectively $t$ and $t'$ and $H$ and $H'$ be their decorations. Let $i$ be an element of $\{1,\dots, |V_H|\}$ such that $(D4)$ is true for $t, H$ and $i$, and $i'$ be an element of $\{1,\dots, |V_{H'}|\}$ such that $(D4)$ is true for $t', H'$ and $i'$. Consider $M$ (resp.~$M'$) the set of vertices of $G$ whose labels are labels of leaves that belong to the $i$-th tree of $t_r$ (resp.~$i'$-th tree of $t'_{r'}$): $M$ (resp.~$M'$) is a module of $G$. Since the $i$-th tree of $t_r$ (resp.~the $i'$-th tree of $t'_{r'}$) is $(\mathcal{P},\mathcal{P}^{\bullet})$-consistent the subgraph of $G$ induced by $M$ (resp.~$M'$) is $(\mathcal{P},\mathcal{P}^{\bullet})$-consistent. 

We now prove by contradiction that $M=M'$. By symmetry we can assume that $M'\not \subset M$.

First assume that $M\cap M'=\emptyset$. Note that $\mathsf{blo}_{M,1}(\mathsf{blo}_{M'}(G))=\mathsf{blo}_{M',0}(\mathsf{blo}_{M}(G))$. Since $\mathsf{blo}_M(G)=\mathsf{blo}_{\{v_i\}}(H)$ and $\mathsf{blo}_{M'}(G)=\mathsf{blo}_{\{v_{i'}\}}(H')$, we get that $\mathsf{blo}_{M',0}(\mathsf{blo}_{\{v_i\}}(H))=\mathsf{blo}_{M,1}(\mathsf{blo}_{\{v_{i'}\}}(H'))$ which contradicts $(C3)$ as both subgraphs of $G$ induced by $M$ and $M'$ are $(\mathcal{P},\mathcal{P}^{\bullet})$-consistent.

Now assume that $M\cap M'\neq \emptyset$. Let $L$ be the subset of $V_H$ such that $v\in L$ if and only if the $\ell(v)$-th tree of $t_r$ contains a leaf labeled with the label of an element of $M'$. Since $M'$ is a module of $G$ and $M\cap M'\neq \emptyset$, $L$ is a module of $\mathsf{blo}_{\{v_i\}}(H)$ containing the blossom. Since $M'$ is not included in $M$, by $(C5)$, $L=H$. Since $M'\neq G$, there exists a vertex $w$ in $G$ such that $w\not \in M'$. Let $w'$ be the vertex of $H$ such that $w$ is in the $\ell(w')$-th tree of $t_{r}$. Since $M'$ is a module, every vertex of $M'$ is either connected or not to $w$, thus $w'$ is connected to every vertex of $H$ (except $w'$) or to none of them. This means that $H$ is either $\oplus$-decomposable or $\ominus$-decomposable, which is a contradiction.

Thus $M=M'$ and $\mathsf{blo}_{\{v_i\}}(H)=\mathsf{blo}_M(G)=\mathsf{blo}_{M'}(G)=\mathsf{blo}_{\{v_{i'}\}}(H')$, and we get that $H=H'$, and that $i=i'$: thus $i$ is unique.

We know that the $i$-th tree of $t_r$ and the $i$-th tree of $t'_{r'}$ are $(\mathcal{P},\mathcal{P}^{\bullet})$-consistent and the associated graph is the one induced by $M$. By taking the reduction of the trees and the graph, we get by minimality of $t$ that the reductions of both trees are equal. Since $M=M'$, it implies that both subtrees are the same: thus $t=t'$.\end{proof}

\section{Zoology of graph classes with few $P_4$'s}\label{sec:cla}

Several classes have been defined as generalizations of the class of $P_4$-free graphs, the cographs. Here the classes we will focus on are the following: $P_4$-reducible graphs \cite{reduc, reductible}, $P_4$-sparse graphs \cite{sparse, jamison} $P_4$-lite graphs \cite{lite}, $P_4$-extendible graphs \cite{extendible}, $P_4$-tidy graphs \cite{tidy}.\medskip

The aim of this section is to give explicit sets $\mathcal{P}$ and $\mathcal{P}^{\bullet}$ such that $\mathcal{G}_{\mathcal{P},\mathcal{P}^{\bullet}}$ is one of the previously mentioned classes.

\subsection{Basic definitions}
The following results and definitions are from \cite[Section 11.3]{graphclasses}.

\begin{dfn}
A graph $G$ is a \emph{$P_{k}$} if it is a path of $k$ vertices, and a $C_{k}$ if it is a cycle of $k$ vertices.

The two vertices of degree one of a $P_4$ are called the endpoints, the two vertices of degree two are called the midpoints.
\end{dfn}

\begin{nota}
For a graph $G$, we denote by $\overline{G}$ its complementary.
\end{nota}

The modular decompositions of classes of graphs we consider are already well-known \cite{tidy}. To explain the different properties, we need the notion of spider and bull.

\begin{dfn}\label{spider}
\emph{A spider} is a graph $G$, such that there exists a partition of $V_G$ in three parts, $K,S,R$, verifying: 
\begin{itemize}
    \item $|K|\geq 2$;
    \item $K$ induces a clique;
    \item $S$ induces a graph without edges;
    \item every element of $R$ is connected to every element of $K$ but to none of $S$;
    \item there exists a bijection $f$ from $K$ to $S$ such that for every $k\in K$, $k$ is only connected to $f(k)$ in $S$, or such that for every $k\in K$, $k$ is connected to every element of $S$ except $f(k)$. In the first case the spider is called \emph{thin}, in the second one it is called \emph{fat}.
\end{itemize}

\end{dfn}

\begin{figure}[htbp]
\begin{center}
\centering
\includegraphics[scale=0.6]{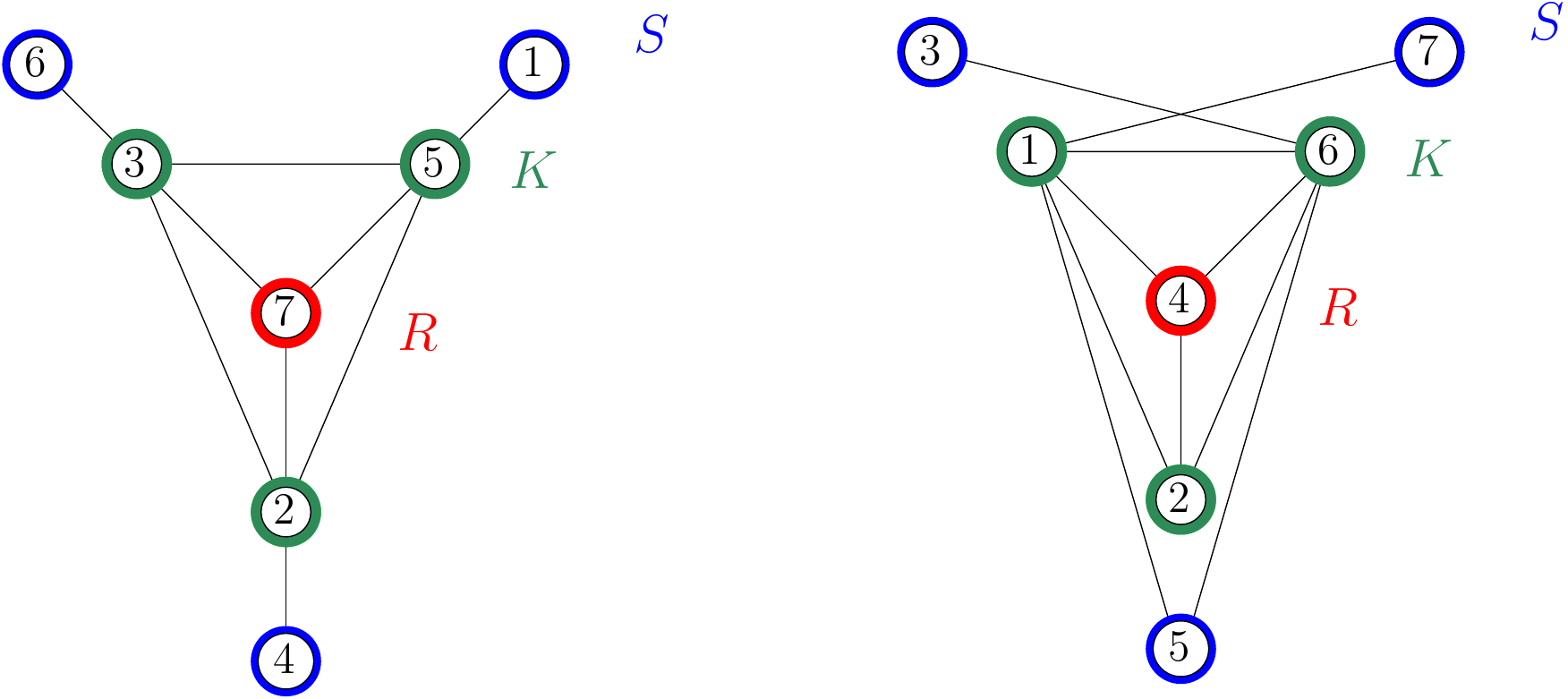}
\caption{Left: a thin spider. Right: a fat spider. Both with $|K|=3$.}
\end{center}
\end{figure}

\begin{rem}
For every spider $G$, the partition $(K,S,R)$ is uniquely determined by $G$. Moreover, the bijection $f$ given by the definition is unique, except in the case $|K|=2$. In this case, since there is no difference between a thin and a fat spider, a spider with $|K|=2$ is called \emph{thin}. A spider with $|K|=2$ and $|R|=1$ is called a \emph{bull}, and a spider with $|K|=2$ and $|R|=0$ is simply a \emph{$P_4$}.
\end{rem}

\begin{figure}[htbp]
\begin{center}
\centering
\includegraphics[scale=0.6]{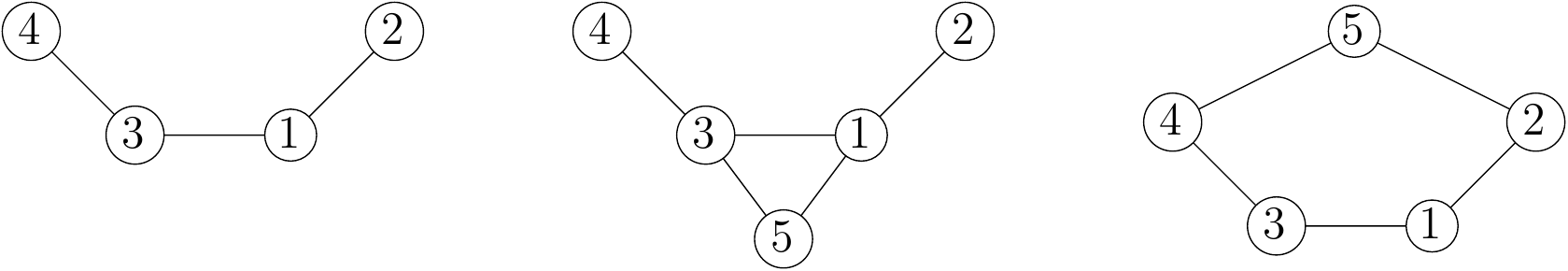}
\caption{From left to right: a $P_4$, a bull, a $C_5$}\label{bull}
\end{center}
\end{figure}

\begin{prop}
A spider is prime if and only if $|R|\leq 1$.
\end{prop}

In the following, if $|R|=1$, the vertex belonging to $R$ will be a blossom of the spider, and it will be its only blossom: such spiders will be called \emph{blossomed spiders}. If $|R|=0$, the spider will have no blossom. This also applies for bulls and $P_4$.

\begin{dfn}\label{pseudo}
We call a graph $H$ a \emph{pseudo-spider} if there exists a prime spider $G$ such that, if we duplicate a vertex that is not a blossom of $G$ (his label is the new number of vertices), and if either by adding or not an edge between the vertex and its duplicate, the graph obtained is a relabeling of $H$. If $|K|=2$, we also call $H$ a \emph{pseudo-$P_4$}.

Moreover, we say that $H$ is a \emph{blossomed pseudo-spider} if $G$ is a blossomed spider. If $|K|=2$, we also call $H$ a \emph{pseudo-bull}.
\end{dfn}

\begin{figure}[htbp]
\begin{center}
\centering
\includegraphics[scale=0.6]{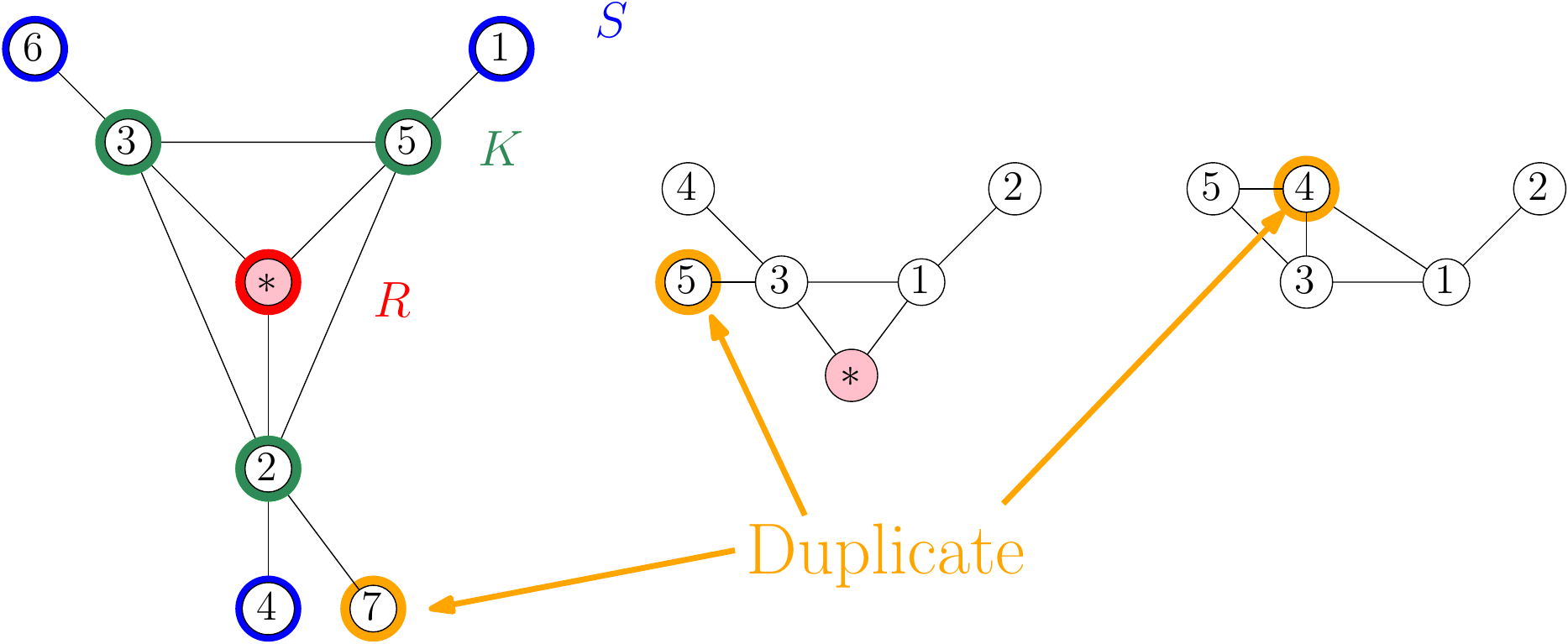}
\caption{A blossomed pseudo-spider, a pseudo-bull, a pseudo $P_4$}
\end{center}
\end{figure}

\begin{lem}
A prime spider with $0$ or $1$ blossom has $|K|!$ automorphisms (as there is a natural bijection between the automorphisms of the spider and the automorphisms of $K$).

A pseudo-spider with $0$ or $1$ blossom has $2\times (|K|-1)!$ automorphisms.

\end{lem}

\subsection{$P_4$-tidy graphs}

\begin{dfn}
A graph $G$ is said to be a \emph{$P_4$-tidy graph} if, for every subgraph $H$ of $G$ inducing a $P_4$, there exists at most one vertex $y\in V_G\backslash V_H$ such that $y$ is connected to at least one element of $H$ but not all, and $y$ is not connected to exactly both midpoints of $H$.
\end{dfn}

\begin{thm}\label{ttidy}

Let $\mathcal{P}_{\mathrm{tidy}}$ be the set containing all $C_5$, $P_5$, $\overline{P_5}$, all prime spiders without blossom and all pseudo-spiders without blossom. Let $\mathcal{P}^{\bullet}_{\mathrm{tidy}}$ be the set of all blossomed prime spiders and all blossomed pseudo-spiders. Then the set of graphs that are $P_4$-tidy is $\mathcal{G}_{\mathcal{P}_{\mathrm{tidy}},\mathcal{P}^{\bullet}_{\mathrm{tidy}}}$.

\end{thm}

\begin{proof}
It is simply a reformulation in our setting of \cite[Theorem 3.3]{tidy} that states that a graph $G$ is $P_4$-tidy if and only if its canonical tree $t$ verifies the following conditions: 

\begin{itemize}
    \item Every node in $t$ is labeled with $\oplus$, $\ominus$, $C_5$, $P_5$, $\overline{P_5}$ or a prime spider.
    \item If a node $w$ in $t$ is decorated with $C_5$, $P_5$ or $\overline{P_5}$, every element of $t_w$ is reduced to a single leaf.
    \item If a node $w$ in $t$ is decorated with a prime spider with $|R|=0$, every element of $t_w$ is a tree of size at most two, and at most one is of size two.
    \item If a node $w$ in $t$ is decorated with a prime spider $H$ with $|R|=1$, let $v$ be the vertex of $H$ in $R$, and $t'$ the $\ell(v)$-th tree of $t_w$. Every element of $t_w\backslash \{t'\}$ is a tree of size at most two, and at most one is of size two.\qedhere
    \end{itemize}\end{proof}

\begin{prop}\label{p1}
The pair $(\mathcal{P}_{\mathrm{tidy}},\mathcal{P}^{\bullet}_{\mathrm{tidy}})$ verifies $(C)$
\end{prop}

\begin{proof}
Note that all the graph in $\mathcal{P}_{\mathrm{tidy}}$ or $\mathcal{P}^{\bullet}_{\mathrm{tidy}}$ are prime except the pseudo-spiders. The only modules of the pseudo-spiders are the trivial ones, and the module formed by the vertex that was duplicated and its duplicate, which implies $(C5)$.\\
$(C2)$ is also verified with the previous observation, as the modules of every graph in $\mathcal{P}_{\mathrm{tidy}}$ are trivial.\\
$(C1)$ is clearly verified and $(C4)$ can be checked easily as all the graphs in $\mathcal{P}_{\mathrm{tidy}}\cup \mathcal{P}^{\bullet}_{\mathrm{tidy}}$ are connected, and their complementary is also connected.

For $(C3)$, assume that for $(F,F')^2\in \mathcal{P}^{\bullet}_{\mathrm{tidy}}$ and $M,M'$ are respectively flowerless modules of $F$ and $F'$, $\mathsf{blo}_{M,0}(G)=\mathsf{blo}_{M',1}(G)$. By cardinality argument, $F$ and $F'$ are either both spiders, or both pseudo-spiders of same size. If both are spiders, as $R$ is uniquely determined by the spiders, and the only element of $R$ does not have the same label in $\mathsf{blo}_{M,0}(G)$ and in $\mathsf{blo}_{M,1}(G)$, we get a contradiction. If both are pseudo-spiders, note that the original node and its duplicate form the only module of size $2$ of $\mathsf{blo}_{M,0}(G)$. Thus the only element of $R$ (in the original spiders) is uniquely determined by the pseudo-spiders, and the only element of $R$ does not have the same label in $\mathsf{blo}_{M,0}(G)$ and in $\mathsf{blo}_{M,1}(G)$, we get a contradiction. \end{proof}

\subsection{$P_4$-lite graphs}

\begin{dfn}
A graph $G$ is said to be a \emph{$P_4$-lite graph} if every subgraph of $G$ of size at most $6$ does not contain three induced $P_4$.
\end{dfn}

\begin{thm}

Let $\mathcal{P}_{\mathrm{lite}}$  be the set containing all $P_5$, $\overline{P_5}$, all prime spiders without blossom and all pseudo-spiders without blossom. Let $\mathcal{P}^{\bullet}_{\mathrm{lite}}$ to be the set containing all blossomed prime spiders and all blossomed pseudo-spiders.
Then the set of graphs that are $P_4$-lite is $\mathcal{G}_{\mathcal{P}_{\mathrm{lite}},\mathcal{P}^{\bullet}_{\mathrm{lite}}}$.

\end{thm}

\begin{proof}
It is simply a reformulation in our setting of \cite[Theorem 3.8]{tidy} that states that a graph $G$ is $P_4$-lite if and only if its canonical tree $t$ verifies the following conditions: 

\begin{itemize}
    \item Every node in $t$ is labeled with $\oplus$, $\ominus$, $P_5$, $\overline{P_5}$ or a prime spider.
    \item If a node $w$ in $t$ is decorated with $P_5$ or $\overline{P_5}$, every element of $t_w$ is reduced to a single leaf.
    \item If a node $w$ in $t$ is decorated with a prime spider with $|R|=0$, every element of $t_w$ is a tree of size at most two, and at most one is of size two.
    \item If a node $w$ in $t$ is decorated with a prime spider $H$ with $|R|=1$, let $v$ be the vertex of $H$ in $R$, and $t'$ the $\ell(v)$-th tree of $t_w$. Every element of $t_w\backslash \{t'\}$ is a tree of size at most two, and at most one is of size two.\qedhere
\end{itemize}\end{proof}

By \Cref{p1} since $\mathcal{P}_{\mathrm{lite}}\subset \mathcal{P}_{\mathrm{tidy}}$, $\mathcal{P}^{\bullet}_{\mathrm{lite}}\subset \mathcal{P}^{\bullet}_{\mathrm{tidy}}$ we get that the pair $(\mathcal{P}_{\mathrm{lite}},\mathcal{P}^{\bullet}_{\mathrm{lite}})$ verifies $(C)$.

\subsection{$P_4$-extendible graphs}

\begin{dfn}
A graph $G$ is said to be a \emph{$P_4$-extendible graph} if, for every subgraph $H$ of $G$ inducing a $P_4$, there exists at most one vertex $y\in V_G\backslash V_H$ such that $y$ belongs to an induced $P_4$ sharing at least one vertex with $H$.
\end{dfn}

\begin{thm}

Let $\mathcal{P}_{\mathrm{ext}}$ be the set containing all $C_5$, $P_5$, $\overline{P_5}$, $P_4$ and all pseudo-$P_4$. Let $\mathcal{P}^{\bullet}_{\mathrm{ext}}$ be the set containing all bulls and all pseudo-bulls.
Then the set of graphs that are $P_4$-extendible is $\mathcal{G}_{\mathcal{P}_{\mathrm{ext}},\mathcal{P}^{\bullet}_{\mathrm{ext}}}$.

\end{thm}

\begin{proof}
It is simply a reformulation in our setting of \cite[Theorem 3.7]{tidy} that states that a graph $G$ is $P_4$-extendible if and only if its canonical tree $t$ verifies the following conditions: 

\begin{itemize}
    \item Every node in $t$ is labeled with $\oplus$, $\ominus$, $C_5$, $P_5$, $\overline{P_5}$, $P_4$ or a bull.
    \item If a node $w$ in $t$ is decorated with $C_5$, $P_5$ or $\overline{P_5}$, every element of $t_w$ is reduced to a single leaf.
    \item If a node $w$ in $t$ is decorated with $P_4$, every element of $t_w$ is a tree of size at most two, and at most one is of size two.
    \item If a node $w$ in $t$ is decorated with a bull $G$, let $v$ be the vertex of $G$ in $R$, and $t'$ the $\ell(v)$-th tree of $t_w$. Every element of $t_w\backslash \{t'\}$ is a tree of size at most two, and at most one is of size two.\qedhere
\end{itemize}\end{proof}

By \Cref{p1} since $\mathcal{P}_{\mathrm{ext}}\subset \mathcal{P}_{\mathrm{tidy}}$, $\mathcal{P}^{\bullet}_{\mathrm{ext}}\subset \mathcal{P}^{\bullet}_{\mathrm{tidy}}$ we get that the pair $(\mathcal{P}_{\mathrm{ext}},\mathcal{P}_{\mathrm{ext}}^{\bullet})$ verifies $(C)$.

\subsection{$P_4$-sparse graphs}

\begin{dfn}
A graph $G$ is said to be a \emph{$P_4$-sparse graph} if every subgraph of $G$ of size $5$ does not contain two induced $P_4$.
\end{dfn}

\begin{thm}

Let $\mathcal{P}$ be the set containing all prime spiders without blossom. Let $\mathcal{P}^{\bullet}$ be the set containing all blossomed prime spiders.
Then the set of graphs that are $P_4$-sparse is $\mathcal{G}_{\mathcal{P},\mathcal{P}^{\bullet}}$.

\end{thm}

\begin{proof}

It is simply a reformulation in our setting of \cite[Theorem 3.4]{giakou} that states that a graph $G$ is $P_4$-sparse if and only if its canonical tree $t$ verifies the following conditions: 

\begin{itemize}
    \item Every node in $t$ is labeled with $\oplus$, $\ominus$ or a prime spider.
    
    \item If a node $w$ in $t$ is decorated with a prime spider with $|R|=0$, every element of $t_w$ is reduced to a single leaf.
    \item If a node $w$ in $t$ is decorated with a prime spider $h$ with $|R|=1$, let $v$ be the vertex of $H$ in $R$, and $t'$ the $\ell(v)$-th tree of $t_w$. Every element of $t_w\backslash \{t'\}$ is reduced to a single leaf.\qedhere
\end{itemize}\end{proof}

By \Cref{p1} since $\mathcal{P}_{\mathrm{spa}}\subset \mathcal{P}_{\mathrm{tidy}}$, $\mathcal{P}^{\bullet}_{\mathrm{spa}}\subset \mathcal{P}^{\bullet}_{\mathrm{tidy}}$ we get that the pair $(\mathcal{P}_{\mathrm{spa}},\mathcal{P}^{\bullet}_{\mathrm{spa}})$ verifies $(C)$.

\subsection{$P_4$-reducible graphs}

\begin{dfn}
A graph $G$ is said to be a \emph{$P_4$-reducible graph} if every vertex of $G$ belongs to at most one induced $P_4$.
\end{dfn}

\begin{thm}

Let $\mathcal{P}_{\mathrm{red}}$ be the set containing all $P_4$. Let $\mathcal{P}^{\bullet}_{\mathrm{red}}$ be the set containing all bulls.
Then the set of graphs that are $P_4$-reducible is $\mathcal{G}_{\mathcal{P}_{\mathrm{red}},\mathcal{P}^{\bullet}_{\mathrm{red}}}$.

\end{thm}

\begin{proof}
It is simply a reformulation in our setting of \cite[Theorem 4.2]{giakou} that states that a graph $G$ is $P_4$-reducible if and only if its canonical tree $t$ verifies the following conditions: 

\begin{itemize}
    \item Every node in $t$ is labeled with $\oplus$, $\ominus$, $P_4$ or a bull.
    
    \item If a node $w$ in $t$ is decorated with a $P_4$, every element of $t_w$ is reduced to a single leaf.
    
    \item If a node $w$ in $t$ is decorated with a bull $H$, let $v$ be the vertex of $H$ in $R$, and $t'$ the $\ell(v)$-th tree of $t_w$. Every element of $t_w\backslash \{t'\}$ is reduced to a single leaf.\qedhere
\end{itemize}\end{proof}

By \Cref{p1} since $\mathcal{P}_{\mathrm{red}}\subset \mathcal{P}_{\mathrm{tidy}}$, $\mathcal{P}^{\bullet}_{\mathrm{red}}\subset \mathcal{P}^{\bullet}_{\mathrm{tidy}}$ we get that the pair $(\mathcal{P}_{\mathrm{red}},\mathcal{P}^{\bullet}_{\mathrm{red}})$ verifies $(C)$.

\subsection{$P_4$-free graphs (cographs)}

\begin{dfn}
A graph $G$ is said to be a \emph{cograph} if no subgraph of $G$ induces a $P_4$.
\end{dfn}

\begin{thm}
Set $\mathcal{P}_{\mathrm{cog}}=\emptyset$ and $\mathcal{P}^{\bullet}_{\mathrm{cog}}=\emptyset$.
Then the set of graphs that are cographs is $\mathcal{G}_{\mathcal{P}_{\mathrm{cog}},\mathcal{P}^{\bullet}_{\mathrm{cog}}}$.

\end{thm}

\begin{proof}
It is simply a reformulation in our setting of \cite[Theorem 7]{corn} that states that a graph $G$ is a cograph if and only if its canonical tree $t$ has no internal node decorated with a prime graph. \end{proof}

Clearly the pair $(\mathcal{P}_{\mathrm{cog}},\mathcal{P}^{\bullet}_{\mathrm{cog}})$ verifies $(C)$.

\section{Enriched modular decomposition: enumerative results}\label{sec4}
\subsection{Exact enumeration}\label{sec:exa}
In the following, we establish combinatorial identities between formal power series involving subsets of $\mathcal{P}$ and $\mathcal{P}^{\bullet}$.

Throughout this section, we consider generic pairs $(\mathcal{P},\mathcal{P}^{\bullet})$ where $\mathcal{P}$ (resp.~$\mathcal{P}^{\bullet}$) is a set of graphs with no blossom (resp.~with one blossom) verifying condition $(C)$ defined p.\pageref{cond1}.

Recall that for a graph $G$ with blossoms, $N(G)$ is the number of vertices that are not a blossom: this will be the crucial parameter in the subsequent analysis. Let $P^{\bullet}(z):=\sum\limits_{s\in \mathcal{P}^{\bullet}}\frac{z^{N(s)}}{N(s)!}$ and $P(z):=\sum\limits_{s\in \mathcal{P}}\frac{z^{N(s)}}{N(s)!}$.

For $n\in \N$, let $\mathcal{P}_n$ (resp.~$\mathcal{P}^{\bullet}_n$) be the set of graphs $G$ in $\mathcal{P}$ (resp.~$\mathcal{P}^{\bullet}$) such that $N(G)=n$.

Note that, if both $\mathcal{P}$ and $\mathcal{P}^{\bullet}$ are stable under relabeling (which is the case for the classes of graphs mentioned in \Cref{sec:cla}), for each $n\in \N$, there is a natural action $\Phi_n$ of the permutations of $\{1,\dots, n\}$ over $\mathcal{P}_n$ and $\mathcal{P}^{\bullet}_n$. Let $R_{\mathcal{P}_n}$ and $R_{\mathcal{P}^{\bullet}_n}$ be a system of representants of every orbit under this action, then
$$P^{\bullet}(z)=\sum_{n\in\N}|\mathcal{P}^{\bullet}_n|\frac{z^{n}}{n!}=\sum_{n\in\N}\sum_{s\in R_{\mathcal{P}^{\bullet}_n}}|R_{\mathcal{P}^{\bullet}_n}|\frac{n!}{|\mathrm{Aut}(s)|}\frac{z^{n}}{n!}=\sum_{n\in\N}\sum_{s\in R_{\mathcal{P}^{\bullet}_n}}| R_{\mathcal{P}^{\bullet}_n}|\frac{z^{n}}{|\mathrm{Aut}(s)|}.$$

Similarly, we have:
$$P(z)=\sum_{n\in\N}\sum_{s\in R_{\mathcal{P}_n}}| R_{\mathcal{P}_n}|\frac{z^n}{|\mathrm{Aut}(s)|}.$$

\begin{thm}

For each graph class introduced in \Cref{sec:cla}, we have the following expressions for $P$ and $P^{\bullet}$: \medskip

\begin{center}
\begin{tabular}{|l|l|}
  \hline
  $P_4$-tidy  & $P^{\bullet}_{\mathrm{tidy}}(z)=(2+4z^3)\exp(z^2)-2-2z^2-4z^3-\frac{z^4}{2}-2z^5$ \\
   & $P_{\mathrm{tidy}}(z)=P^{\bullet}_{\mathrm{tidy}}(z)+z^5+\frac{z^5}{10}$ \\
   \hline
  $P_4$-lite  & $P^{\bullet}_{\mathrm{lite}}(z)=(2+4z^3)\exp(z^2)-2-2z^2-4z^3-\frac{z^4}{2}-2z^5$  \\
   & $P_{\mathrm{lite}}(z)=P^{\bullet}_{\mathrm{lite}}(z)+z^5$ \\
   \hline
  $P_4$-extendible  & $P^{\bullet}_{\mathrm{ext}}(z)=\frac{z^4}{2}+2z^5$ \\
  & $P_{\mathrm{ext}}(z)=P^{\bullet}_{\mathrm{ext}}(z)+z^5+\frac{z^5}{10}$ \\
  \hline
  $P_4$-sparse & $P^{\bullet}_{\mathrm{spa}}(z)=P_{\mathrm{spa}}(z)=2(\exp(z^2)-1-z^2-\frac{z^4}{4})$ \\
  \hline
  $P_4$-reducible & $P^{\bullet}_{\mathrm{red}}(z)=P_{\mathrm{red}}(z)=\frac{z^4}{2}$  \\
  \hline
  $P_4$-free & $P^{\bullet}_{\mathrm{cog}}(z)=P_{\mathrm{cog}}(z)=0$  \\
  \hline
\end{tabular}
\end{center}

\end{thm}

\begin{proof}
We only detail the computation of $P_{\mathrm{tidy}}$ and $P^{\bullet}_{\mathrm{tidy}}$ for $P_4$-tidy graphs as this is the most involved case. According to \Cref{ttidy}, $\mathcal{P}_{\mathrm{tidy}}$ is composed of one $C_5$ that has $10$ automorphisms and all its relabelings, one $P_5$, and one $\overline{P_5}$ that both have $2$ automorphisms and all their relabelings.

For $k\geq 3$ (resp.~$k=2$), there are thin and fat spiders corresponding to the $2$ (resp.~$1$) different orbits of the action $\Phi_{2k}$ over prime spiders of size $2k$, each having $k!$ automorphisms. 

For $k\geq 3$ (resp.~$k=2$), there are thin and fat pseudo-spiders, the duplicated vertex can come from $K$ or $S$, and can be connected or not to the initial vertex. These $8$ (resp.~$4$) cases correspond to the $8$ (resp.~$4$) different orbits of the action $\Phi_{2k+1}$ over pseudo-spiders of size $2k+1$, each having $2(k-1)!$ automorphisms.

Thus we have 
\begin{align*}
P_{\mathrm{tidy}}(z)&=\frac{z^5}{10}+\frac{2z^5}{2}+\frac{z^4}{2}+2\sum\limits_{k\geq 3}\frac{z^{2k}}{k!}+4\frac{z^5}{2}+8\sum\limits_{k\geq 3}\frac{z^{2k+1}}{2 (k-1)!}\\
&=z^5+\frac{z^5}{10}+(2+4z^3)\exp(z^2)-2-2z^2-4z^3-\frac{z^4}{2}-2z^5.
\end{align*}

Now let us compute $P^{\bullet}_{\mathrm{tidy}}$. For $k\geq 3$ (resp.~$k=2$), there are thin and fat spiders with blossom corresponding to the $2$ (resp.~$1$) different orbits of the action $\Phi_{2k}$ over blossomed prime spiders $G$ with $2k$ non blossomed vertices, each having $k!$ automorphisms. 

For $k\geq 3$ (resp.~$k=2$), there are thin and fat pseudo-spiders, the duplicated vertex can come from $K$ or $S$, and can be connected or not to the initial vertex. These $8$ (resp.~$4$) cases correspond to the $8$ (resp.~$4$) different orbits of the action $\Phi_{2k+1}$ over blossomed pseudo-spiders  with $2k+1$ non blossomed vertices, each having $2(k-1)!$ automorphisms. 

Hence
$$P^{\bullet}_{\mathrm{tidy}}(z)=\frac{z^4}{2}+2\sum\limits_{k\geq 3}\frac{z^{2k}}{k!}+4\frac{z^5}{2}+8\sum\limits_{k\geq 3}\frac{z^{2k+1}}{2 (k-1)!}=(2+4z^3)\exp(z^2)-2-2z^2-4z^3-\frac{z^4}{2}-2z^5,$$
which gives the announced result.\end{proof}

Let $T$ be the exponential generating function of $\mathcal{T}_{\mathcal{P},\mathcal{P}^{\bullet}}$, the set of trees defined in \cref{d1} counted by their number of leaves. Denote by $\mathcal{T}_{\mathrm{not}\oplus}$ (resp.~$\mathcal{T}_{\mathrm{not}\ominus}$) the set of all $t\in \mathcal{T}_{\mathcal{P},\mathcal{P}^{\bullet}}$ whose root is not decorated with $\oplus$ (resp.~$\ominus$) and by $T_{\mathrm{not}\oplus}$ (resp.~$T_{\mathrm{not}\ominus}$) the corresponding exponential generating function.

\begin{thm}

The exponential generating function $T_{\mathrm{not}\oplus}$ verifies the following equation:
\begin{align}
  \label{ecu123}
T_{\mathrm{not}\oplus}&=z+P+(\exp(T_{\mathrm{not}\oplus})-1)P^{\bullet}+ \exp(T_{\mathrm{not}\oplus})-1-T_{\mathrm{not}\oplus},
\end{align}
and the series $T$ and $T_{\mathrm{not}\ominus}$ are simply given by the following equations:
\begin{align}
  \label{ecu124}
&T=\exp(T_{\mathrm{not}\oplus})-1\\
    \label{ecutri}
 &T_{\mathrm{not}\ominus}=T_{\mathrm{not}\oplus}
\end{align}

Moreover, \cref{ecu123} with $T_{\mathrm{not}\oplus}(0)=0$ determines uniquely (as a formal series) the generating function $T_{\mathrm{not}\oplus}$.
\end{thm}

\begin{proof}
Note that there is a natural involution on $\mathcal{T}_{\mathcal{P},\mathcal{P}^{\bullet}}$: the decoration of every linear node can be changed to its opposite: $\oplus$ to $\ominus$, and $\ominus$ to $\oplus$. Therefore $T_{\mathrm{not}\oplus}=T_{\mathrm{not}\ominus}$.

First, we prove that
\begin{align}
  \label{ecu125}
T&=z+T\times P^{\bullet}+P+2\times (\exp(T_{\mathrm{not}\oplus})-1-T_{\mathrm{not}\oplus})
\end{align}

We split the enumeration of the trees  $t\in \mathcal{T}_{\mathcal{P},\mathcal{P}^{\bullet}}$ according to the different cases of \Cref{d1}.

\begin{enumerate}
    \item[(D1)] The tree $t$ is a single leaf (which gives the $z$ in \cref{ecu125}).
    
    \item[(D2)]  The tree $t$ has a root decorated with a graph $H$ belonging to $\mathcal{P}$. The exponential generating function for a fixed $H$ is $\frac{z^{N(H)}}{N(H)!}$. Summing over all $H$ and all $n$ gives the term $P$ in \cref{ecu125}.
   
   \item[(D3)] The tree $t$ has a root $r$ decorated with $\oplus$ and having $k$ children with $k\geq 2$. In this case, the generating function of the set of the $k$ trees of $t_r$ is $\frac{T_{\mathrm{not}\oplus}^k}{k!}$. Summing over all $k$ implies that the exponential generating function of all trees in case $(D3)$ with a root labeled $\oplus$ is $\exp(T_{\mathrm{not}\oplus})-1-T_{\mathrm{not}\oplus}$.

   The tree $t$ can also have a root $r$ decorated with $\ominus$. Since $T_{\mathrm{not}\oplus}=T_{\mathrm{not}\ominus}$, the exponential generating function of all trees in case $(D3)$ with a root labeled $\ominus$ is $\exp(T_{\mathrm{not}\oplus})-1-T_{\mathrm{not}\oplus}$.
    
  \item [(D4)]

   The tree $t$ has a root $r$ decorated with a graph $H$ and there exists $v\in V_H$ such that $\mathsf{blo}_v(H)=W$ where $W\in \mathcal{P}^{\bullet}$. Denote $t'$ the $\ell(v)$-th tree of $t_r$.
 
    The exponential generating function corresponding to the set of leaves in $t\backslash t'$ is $\frac{z^{N(W)}}{N(W)!}$, and the exponential generating function corresponding to $t'$ is $T$. Note that the tree $t$ is uniquely determined by $W$, the labeled product of $t'$ and the set of leaves of $t\backslash t'$. Thus the corresponding generating function for a fixed $W$ is $T\times \frac{z^{N(W)}}{N(W)!}$. Summing over all $W$ and all $n$ gives the term $T\times P^{\bullet}$ in \cref{ecu125}.

\end{enumerate}

Summing all terms gives \cref{ecu125}.

Similarly, we get
\begin{align}
  \label{ecu126}
T_{\mathrm{not} \oplus}&=z+T\times P^{\bullet}+P+ \exp(T_{\mathrm{not}\oplus})-1-T_{\mathrm{not}\oplus}.
\end{align}

Substracting \cref{ecu126} to \cref{ecu125} gives \cref{ecu124}. Then \cref{ecu123} is an easy consequence from \cref{ecu124,ecu126}.\smallskip

Note that \cref{ecu123} can be rewritten as:
\begin{align}
  \label{eq:equation55555}
T_{\mathrm{not}\oplus}&=z+P+\sum\limits_{k\geq 1}\frac{T_{\mathrm{not}\oplus}^k}{k!}P^{\bullet}+\sum\limits_{k\geq 2}\frac{T_{\mathrm{not}\oplus}^k}{k!}.
\end{align}

For every $n\geq 1$, the coefficient of degree $n$ of $T_{\mathrm{not}\oplus}$ only depends on coefficients of lower degree as $P^{\bullet}(z)$ has no term of degree $0$ or $1$ and $T_{\mathrm{not}\oplus}(0)=0$. Thus \cref{ecu123} combined with $T_{\mathrm{not}\oplus}(0)=0$ determines uniquely $T_{\mathrm{not}\oplus}$.\end{proof}

We are going to define the notions of trees with marked leaves, and of blossomed trees, which will be crucial in the next section. We insist on the fact that the size parameter will count the number of leaves \textbf{including the marked ones but not the blossoms}.

\begin{dfn}\label{mark}
A marked tree is a pair $(t,\mathfrak{I})$ where $t$ is a tree and $\mathfrak{I}$ a partial injection from the set of labels of leaves of $t$ to $\N$. The number of marked leaves is the size of the domain of $\mathfrak{I}$ denoted by $|(t,\mathfrak{I})|$, and a leaf is marked if its label $j$ is in the domain, its mark being $\mathfrak{I}(j)$.
\end{dfn}

\begin{rem}
In the following, we will consider marked trees $(t,\mathfrak{I})$, and subtrees $t'$ of $t$. The marked tree $(t',\mathfrak{I})$ will refer to the marked tree $(t',\mathfrak{I}')$ where $\mathfrak{I}'$ is the restriction of $\mathfrak{I}$ to the set of labels of leaves of $t'$.
\end{rem}

\begin{rem}
Let $\mathcal{F}\in\{\mathcal{T}_{\mathcal{P},\mathcal{P}^{\bullet}},\mathcal{T}_{\mathrm{not}\ominus},\mathcal{T}_{\mathrm{not}\oplus}\}$, and $F$ be its generating exponential function. The exponential generating function of trees in $\mathcal{F}$ with a marked leaf is $zF'(z)$: if there are $f_n$ trees of size $n$ in $\mathcal{F}$, there are $nf_n$ trees with a marked leaf. Thus the generating exponential function is $\sum \limits_{n\geq 1}\frac{n f_n}{n!}z^{n}=zF'(z)$.
\end{rem}

\subsubsection*{Blossoming transformation}

Let $t$ be a tree not reduced to a leaf in $\mathcal{T}_{\mathcal{P},\mathcal{P}^{\bullet}}$, $\ell$ a leaf of $t$ and $n$ the parent of $\ell$. If $n$ is a linear node,
we replace the label of $\ell$ by $*$, and do the reduction on $t$.
If $v$ is a non-linear node, and $\ell$ is in the $i$-th tree of $t_n$ (where $i$ is the element such that $(D4)$ holds in \Cref{d1}), we replace the label of $\ell$ by $*$ and $i$ by $*$ in the decoration of $n$, and do the reduction on both $t$ and the decoration of $v$. If $t$ is reduced to a leaf, we replace the leaf by a blossom. We call such this transformation the \emph{blossoming} of $(t,\ell)$.

We extend this operation to internal node: if $n$ is a internal node, we replace $t[n]$ by its leaf of smallest label, and do the blossoming operation on the tree obtained. The resulting tree is still called the \emph{blossoming} of $(t,n)$.

\begin{dfn}[Blossomed tree]\label{dfnblossom}

A \emph{blossomed tree} is a tree that can be obtained by the blossoming of a tree in $\mathcal{T}_{\mathcal{P},\mathcal{P}^{\bullet}}$. Its size is its number of leaves without blossom. 

A blossom is \emph{$\oplus$-replaceable} (resp.~\emph{$\ominus$-replaceable}) if its parent is not decorated with $\oplus$ (resp.~$\ominus$).

\end{dfn}

\begin{rem}
Similarly to a tree, a blossomed tree can be marked by a partial injection $\mathfrak{I}$.
\end{rem}

We will denote $\mathcal{T}^b$ and $\mathcal{T}_a^b$ with $a\in\{\mathrm{not}\oplus,\mathrm{not}\ominus\}$, and $b\in\{\oplus,\ominus,\mathsf{blo}\}$ the set of trees whose root is not $\oplus$ (resp.~$\ominus$) if $a=\mathrm{not}\oplus$ (resp.~$a=\mathrm{not}\ominus$), and with one blossom that is $b$-replaceable if $b=\oplus$ or $\ominus$, or just with one blossom if $b=\mathsf{blo}$.

We define $T^b$ and $T_a^b$ to be the corresponding exponential generating functions of trees, counted by the number of non blossomed leaves.

However, we take the convention that $T_{\mathrm{not}\oplus}^{\oplus}(0)=0=T_{\mathrm{not}\ominus}^{\ominus}$. In other words, a single leaf is neither in $\mathcal{T}_{\mathrm{not}\oplus}^{\oplus}$ nor in $\mathcal{T}_{\mathrm{not}\ominus}^{\ominus}$. The other series have constant coefficient $1$.

\begin{rem}
From the previously defined involution, it follows that $T_{\mathrm{not}\oplus}^{\ominus}=T_{\mathrm{not}\ominus}^{\oplus}$, $T_{\mathrm{not}\oplus}^{\oplus}=T_{\mathrm{not}\ominus}^{\ominus}$ et $T^{\oplus}=T^{\ominus}$ and $T_{\mathrm{not}\oplus}^{\mathsf{blo}}=T_{\mathrm{not}\ominus}^{\mathsf{blo}}$.
\end{rem}

\begin{thm}\label{4.5}
The functions $T^{\oplus}, T_{\mathrm{not}\oplus}^{\oplus}, T_{\mathrm{not}\ominus}^{\oplus}$ are given by the following equations: 

\begin{align}
\label{b1}
&T^{\oplus}=\frac{1}{2-\exp(T_{\mathrm{not}\oplus})-P^{\bullet}\exp(T_{\mathrm{not}\oplus})}\\
\label{b2}
&T_{\mathrm{not}\oplus}^{\ominus}=\frac{T^{\oplus}}{\exp(T_{\mathrm{not}\oplus})}\\
\label{b3}
&T_{\mathrm{not}\oplus}^{\oplus}=\frac{T^{\oplus}-1}{\exp(T_{\mathrm{not}\oplus})}\
\end{align}

\end{thm}

\begin{figure}[htbp]
\begin{center}
\centering
\includegraphics[scale=0.6]{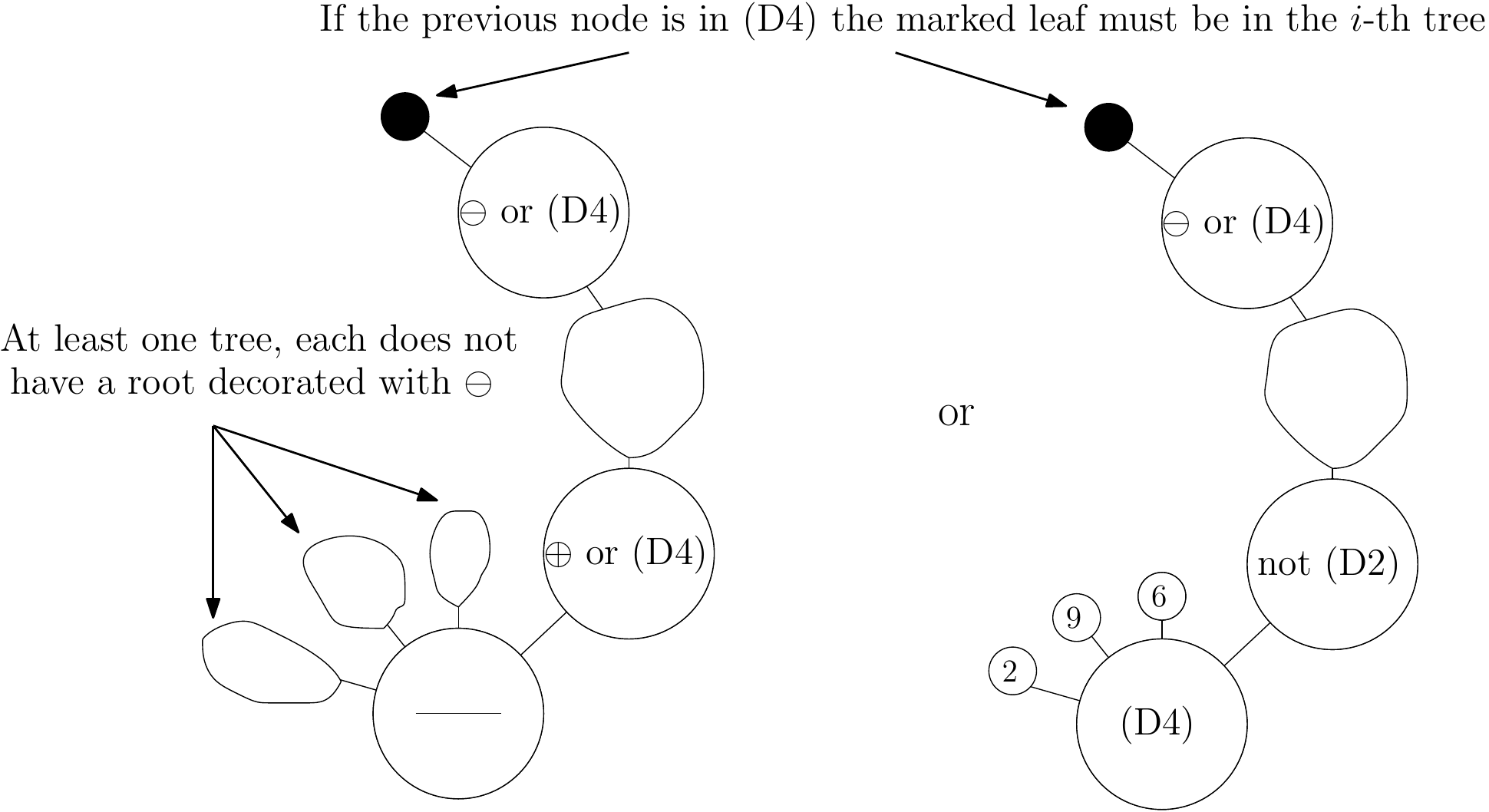}
\caption{Illustration of both cases in the proof of \Cref{4.5}}
\end{center}
\end{figure}

\begin{proof}
Let $t$ be a tree in $\mathcal{T}_{\mathrm{not}\oplus}^{\oplus}$. Note that it cannot be reduced to a single leaf, have a root decorated with $\oplus$ or be in case $(D2)$ of \Cref{d1}.

\begin{enumerate}
    \item[(D3)] The tree $t$ can have a root $r$ decorated with $\ominus$ and having $k$ children with $k\geq 2$. There are $k-1$ subtrees without blossom, and $1$ with a blossom. Thus the generating function of the set of the $k$ trees of $t_r$ is $\frac{T_{\mathrm{not}\ominus}^{k-1}}{(k-1)!}T_{\mathrm{not}\ominus}^{\oplus}$. Summing over all $k$ gives that the exponential generating function of all trees in case $(D3)$ with a root labeled $\ominus$ is $$\sum \limits_{k\geq 2} \frac{T_{\mathrm{not}\ominus}^{k-1}}{(k-1)!}T_{\mathrm{not}\ominus}^{\oplus}=(\exp(T_{\mathrm{not}\ominus})-1)T_{\mathrm{not}\ominus}^{\oplus}$$

     \item[(D4)]

     The tree $t$ can have a root $r$ decorated with $H$ and $v\in V_H$ such that $\mathsf{blo}_v(H)=W$ with $W\in\mathcal{P}^{\bullet}$. Then the blossom must be in the $\ell(v)$-th tree of $t_r$ that will be denoted $t'$.
 
    The exponential generating function corresponding to the set of leaves in $t\backslash t'$ is $\frac{z^{N(W)}}{N(W)!}$, and the exponential generating function corresponding to $t'$ is $T^{\oplus}$. Note that the tree $t$ is uniquely determined by $W$, the labeled product of $t'$ and the set of leaves of $t\backslash t'$.
    Thus the corresponding generating function for a fixed $W$ is $T^{\oplus}\times \frac{z^{N(W)}}{N(W)!}$.
    Summing over all $W$ gives the exponential generating function $T^{\oplus}\times P^{\bullet}$.

\end{enumerate}

This implies the following equation: 
\begin{align}
\label{b4}
&T_{\mathrm{not}\oplus}^{\oplus} =(\exp(T_{\mathrm{not}\ominus})-1)T_{\mathrm{not}\ominus}^{\oplus} +P^{\bullet}T^{\oplus}= (\exp(T_{\mathrm{not}\oplus})-1)T_{\mathrm{not}\oplus}^{\ominus} +P^{\bullet}T^{\oplus}
\end{align}
We have similarly:

\begin{align}
\label{b5}
&T_{\mathrm{not}\oplus}^{\ominus}=1+(\exp(T_{\mathrm{not}\ominus})-1)T_{\mathrm{not}\ominus}^{\ominus}+P^{\bullet}T^{\ominus}=1+(\exp(T_{\mathrm{not}\oplus})-1)T_{\mathrm{not}\oplus}^{\oplus}+P^{\bullet}T^{\oplus}\\
\label{b6666}
&T^{\oplus}=1+(\exp(T_{\mathrm{not}\ominus})-1)T_{\mathrm{not}\ominus}^{\oplus}+(\exp(T_{\mathrm{not}\oplus})-1)T_{\mathrm{not}\oplus}^{\oplus}+P^{\bullet}T^{\oplus}
\end{align}
Thus:
\begin{align}
\label{b6}
    &T^{\oplus}=1+(\exp(T_{\mathrm{not}\oplus})-1)(T_{\mathrm{not}\oplus}^{\oplus}+T_{\mathrm{not}\oplus}^{\ominus})+P^{\bullet}T^{\oplus}
\end{align}

By substracting \cref{b5} to \cref{b6}, we get $T^{\oplus}-T_{\mathrm{not}\oplus}^{\ominus}=(\exp(T_{\mathrm{not}\oplus})-1)T_{\mathrm{not}\oplus}^{\ominus}$ which implies \cref{b2}.

Using \cref{b4,b6}, we get $$T^{\oplus}=1+(\exp(T_{\mathrm{not}\oplus})-1)T_{\mathrm{not}\oplus}^{\oplus}+T_{\mathrm{not}\oplus}^{\oplus}=1+\exp(T_{\mathrm{not}\oplus})T_{\mathrm{not}\oplus}^{\oplus}$$
which implies \cref{b3}.

Multiplying \cref{b4} by $\exp(T_{\mathrm{not}\oplus})$ and using \cref{b3,b2}, it follows that:
$$T^{\oplus}-1=(\exp(T_{\mathrm{not}\oplus})-1)T^{\oplus}+\exp(T_{\mathrm{not}\oplus})P^{\bullet}T^{\oplus}.$$

Thus $T^{\oplus}(2-\exp(T_{\mathrm{not}\oplus})-P^{\bullet}\exp(T_{\mathrm{not}\oplus}))=1$ which implies \cref{b1}.\end{proof}

\begin{thm}\label{4.4}
We also have the following equations: 

\begin{align}
\label{o1}
&T^{\mathsf{blo}}=\frac{\exp(T_{\mathrm{not}\oplus})}{2-\exp(T_{\mathrm{not}\oplus})-P^{\bullet}\exp(T_{\mathrm{not}\oplus})}\\
\label{o2}
&T_{\mathrm{not}\oplus}^{\mathsf{blo}}=\frac{1}{\exp(T_{\mathrm{not}\oplus})}T^{\mathsf{blo}}\
\end{align}

\end{thm}

\begin{proof}
By the same techniques used as those of the previous proof, we establish that:

\begin{align}
\label{o5}
T^{\mathsf{blo}}&=1+2(\exp(T_{\mathrm{not}\oplus})-1)T_{\mathrm{not}\oplus}^{\mathsf{blo}}+P^{\bullet}T^{\mathsf{blo}};\\
\label{o6}
T_{\mathrm{not}\oplus}^{\mathsf{blo}}&=1+(\exp(T_{\mathrm{not}\oplus})-1)T_{\mathrm{not}\oplus}^{\mathsf{blo}}+P^{\bullet}T^{\mathsf{blo}}.
\end{align}

By substracting \cref{o6} to \cref{o5}, we get that: $$T^{\mathsf{blo}}-T_{\mathrm{not}\oplus}^{\mathsf{blo}}= (\exp(T_{\mathrm{not}\oplus})-1)T_{\mathrm{not}\oplus}^{\mathsf{blo}}$$ which implies \cref{o2}. 

By multiplying \cref{o6} by $\exp(T_{\mathrm{not}\oplus})$ and using \cref{o2} we get that: 
$$T^{\mathsf{blo}}\left(2-P^{\bullet}\exp(T_{\mathrm{not}\oplus})-\exp(T_{\mathrm{not}\oplus})\right)=\exp(T_{\mathrm{not}\oplus})$$
which implies \cref{o1}.\end{proof}

Combining \cref{4.5} and \cref{4.4} we obtain:
\begin{cor}
We have the following equations: 
\begin{align}
\label{o3}
&T^{\mathsf{blo}}=\exp(T_{\mathrm{not}\oplus})T^{\oplus};\\
\label{o4}
&T_{\mathrm{not}\oplus}^{\mathsf{blo}}=T^{\oplus}.\
\end{align}

\end{cor}

\subsection{Asymptotic enumeration}\label{sec4.2}
In the following, we derive from the previously obtained equations the radii of the different series introduced, the asymptotic behavior of the different series in $R$ and an equivalent of the number of graphs in $\mathcal{G}_{\mathcal{P},\mathcal{P}^{\bullet}}$

From now on, we assume that $P$ and $P^{\bullet}$ have a positive radius of convergence.
\label{pageP}
Let $R_0$ be the minimum of their radii of convergence. Denote by $P(R_0)$ and $P^{\bullet}(R_0)$ the limit in $[0,+\infty]$ of $P$ and $P^{\bullet}$ at $R_0^-$. 

In the following, we assume that one of the conditions below is verified: 
\begin{itemize}
    \item $P^{\bullet}(R_0)\geq 1$
    \item $R_0+P(R_0)+2\ln(1+P^{\bullet}(R_0))-P^{\bullet}(R_0)>2\ln(2)-1$
\end{itemize}

Note that one of these conditions is verified in the different classes of graphs we study, as $R_0=+\infty$.

Denote by $R$ the only solution in $[0,R_0)$  of the equation: \begin{align}
\label{eqimp}
    R+P(R)+2\ln(1+P^{\bullet}(R))-P^{\bullet}(R)=2\ln(2)-1
    \end{align}
such that $P^{\bullet}(R)<1$  (unicity comes from the fact that $z\mapsto 2\ln(1+z)-z$ is increasing in $[0,1]$). Note that by definition, $0<R<R_0$.

Recall that a formal series $A$ is \emph{aperiodic} if there does not exist two integers $r\geq 0$ and $d\geq 2$ and $B$ a formal series such that $A(z)=z^rB(z^d)$.
\begin{lem}
The functions $T$, $T_{\mathrm{not}\oplus}$, $T^{\oplus}$, $T_{\mathrm{not}\oplus}^{\ominus}$, $T_{\mathrm{not}\oplus}^{\oplus}$, $T^{\mathsf{blo}}$, $T_{\mathrm{not}\oplus}^{\mathsf{blo}}$ are aperiodic.
\end{lem}

\begin{proof}

One can easily check that for each of the previous series, the coefficients of degree $3$ and $4$ are positive, and thus all the series are aperiodic.\end{proof}

\begin{dfn}
A set $\Delta$ is a $\Delta$-domain at $1$ if there exist two positive numbers $R$ and $\frac{\pi}{2}<\phi< \pi$ such that 

$$\Delta=\{z\in\C | |z|\leq R, z\neq 1, |\mathrm{arg}(1-z)|< \phi\}$$

For every $w\in \C^*$, a set is a $\Delta$-domain at $w$ if it is the image of a $\Delta$-domain by the mapping $z\mapsto zw$.
\end{dfn}

\begin{dfn}
A power series $U$ is said to be \emph{$\Delta$-analytic} if it has a positive radius of convergence $\rho$ and there exists a $\Delta$-domain $D$ at $\rho$ such that $U$ has an analytic continuation on $D$.
\end{dfn}

\begin{thm}\label{thm5}
Both $T$ and $T_{\mathrm{not}\oplus}$ have $R$ as radius of convergence and a unique dominant singularity at $R$. They are $\Delta$-analytic. Their asymptotic expansions near $R$ are: 

\begin{align}
\label{ecus}
&T_{\mathrm{not}\oplus}(z)=\ln\left(\frac{2}{1+P^{\bullet}(R)}\right)-\kappa \sqrt{1-\frac{z}{R}}+\mathcal{O}\left(1-\dfrac{z}{R}\right)\\
\label{ecut}
&T(z)=\frac{2}{1+P^{\bullet}(R)}-1-\frac{2}{1+P^{\bullet}(R)}\kappa\sqrt{1-\frac{z}{R}}+\mathcal{O}\left(1-\dfrac{z}{R}\right)
\end{align}

where $\kappa$ is the constant given by:
$$\kappa=\sqrt{R\left(1+P'(R)+\frac{(1-P^{\bullet}(R))(P^{\bullet })'(R)}{1+P^{\bullet}(R)}\right)}$$
\end{thm}

\begin{proof}

We begin with the expansion of $T_{\mathrm{not}\oplus}$ for which we apply the smooth implicit theorem \cite[Theorem VII.3, p.467]{flajolet2009analytic}. Following \cite[Sec VII.4.1]{flajolet2009analytic} we claim that $T_{\mathrm{not}\oplus}$ satisfies the settings of the so-called \emph{smooth implicit-function schema}: $T_{\mathrm{not}\oplus}$ is solution of
$$T=G(z,T),$$
where $G(z,w)=z+P(z)+(\exp(w)-1)P^{\bullet}(z)+(\exp(w)-1-w)$.

The singularity analysis of $T_{\mathrm{not}\oplus}$ will go through the study of the characteristic system:
$$\begin{cases}
&G(r,s)=s,\\
&G_w(r,s)=1
\end{cases}\qquad \text{ with }0<r<R, \ s>0
$$
where $F_x=\frac{\partial F}{\partial x}$.\\
Note that $(r,s)=\left(R,\ln\left(\frac{2}{1+P^{\bullet}(R)}\right)\right)$ is a solution of the characteristic system of $G$ since 
\begin{itemize}
\item $G_w(r,s)=\exp(s)(1+P^{\bullet}(R))-1=2-1=1$
\item $G(r,s)=R+P(R)-P^{\bullet}(R)+\partial_wG(r,s) -s=2\ln(2)-1-2\ln(1+P^{\bullet}(R))+1-s=2s-s=s$
\end{itemize}
Moreover
\begin{itemize}
    \item $G_z(r,s)=1+P'(R)+(\exp(s)-1)(P^{\bullet})'(R)=1+P'(R)+\frac{(1-P^{\bullet}(R))(P^{\bullet})'(R)}{(1+P^{\bullet}(R))}$
    \item $G_{w,w}(r,s)=\exp(s)(1+P^{\bullet}(r))=2$
\end{itemize}

The expansion of $T$ is then a consequence of \cref{ecu124} and of the expansion of $T_{\mathrm{not}\oplus}$.\end{proof}

\begin{cor}
The radius of convergence of $T^{\oplus}$, $T_{\mathrm{not}\oplus}^{\ominus}$, $T_{\mathrm{not}\oplus}^{\oplus}$, $T^{\mathsf{blo}}$, and $T_{\mathrm{not}\oplus}^{\mathsf{blo}}$ is $R$ and $R$ is the unique dominant singularity of these series. They are $\Delta$-analytic and their asymptotic expansions near $R$ are:
\begin{align}
\label{ecuua}
&T^{\oplus}=\frac{1}{2\kappa}\left(1-\frac{z}{R}\right)^{-\frac{1}{2}}+\mathcal{O}(1)\\
\label{ecuub}
&T_{\mathrm{not}\oplus}^{\ominus}=\frac{(1+\mathcal{P}^{\bullet}(R))}{4\kappa}\left(1-\frac{z}{R}\right)^{-\frac{1}{2}}+\mathcal{O}(1)\\
\label{ecuuc}
&T_{\mathrm{not}\oplus}^{\oplus}=\frac{(1+\mathcal{P}^{\bullet}(R))}{4\kappa}\left(1-\frac{z}{R}\right)^{-\frac{1}{2}}+\mathcal{O}(1)\\
\label{eqblo}
&T^{\mathsf{blo}}=\frac{1}{(1+\mathcal{P}^{\bullet}(R))\kappa}\left(1-\frac{z}{R}\right)^{-\frac{1}{2}}+\mathcal{O}(1)\\
\label{ecuud}
&T_{\mathrm{not}\oplus}^{\mathsf{blo}}=\frac{1}{2\kappa}\left(1-\frac{z}{R}\right)^{-\frac{1}{2}}+\mathcal{O}(1)
\end{align}

\end{cor}

\begin{proof}

Note that, if $|z|\leq R$, 
\begin{align*}
    &|(1+P^{\bullet}(z))\exp(T_{\mathrm{not}\oplus}(z))|\leq(1+P^{\bullet}(|z|))\exp(|T_{\mathrm{not}\oplus}(z)|)\leq (1+P^{\bullet}(R))\exp(T_{\mathrm{not}\oplus}(R))=2
\end{align*}
with equality if and only if $z=R$ by aperiodicity from Daffodil lemma \cite[Lemma IV.1]{flajolet2009analytic} and since $T_{\mathrm{not}\oplus}(R)>0$. 

By \cref{thm5} $$2-(1+P^{\bullet}(z))\exp(T_{\mathrm{not}\oplus}(z))=2\kappa \sqrt{1-\frac{z}{R}}+\mathcal{O}\left(1-\frac{z}{R}\right).$$ 

Hence, by compactness, the LHS function can be extended to a $\Delta$-domain $D$ at $R$ with $2-(1+P^{\bullet}(z))\exp(T_{\mathrm{not}\oplus})(z)\neq 0$ for every $z\in D$.

\cref{b1} shows that $T^{\oplus}$ can be extended to $D$ and yields the announced expansions when $z$ tends to $R$. These expansions show that all these series have a radius of convergence exactly equal to $R$.\end{proof}

Applying the Transfer Theorem \cite[Corollary VI.1 p.392]{flajolet2009analytic} to the results of \cref{thm5}, we obtain an equivalent of the number of trees of size $n$ in $\mathcal{T}_{\mathcal{P},\mathcal{P}^{\bullet}}$. Since there is a one-to-one correspondence between graphs in $\mathcal{G}_{\mathcal{P},\mathcal{P}^{\bullet}}$ and trees in $\mathcal{T}_{\mathcal{P},\mathcal{P}^{\bullet}}$, we get the following result:

\begin{cor}\label{cor1}
The number of graphs in $\mathcal{G}_{\mathcal{P},\mathcal{P}^{\bullet}}$ of size $n$ is asymptotically equivalent to 
$$C \frac{n!}{R^nn^{\frac{3}{2}}}\quad \text{where}\quad C=\frac{\kappa}{\sqrt{\pi}(1+P^{\bullet}(R))}.$$
\end{cor}

Here are the numerical approximations of $R$ and $C$ in the different cases: \medskip

\begin{center}
\begin{tabular}{|c|c|c|c|}
  \hline
  class of graph & $R^{-1}$ & $R$ & $C$ \\
  \hline
  $P_4$-tidy  & $2.90405818$ & $0.34434572$ & $0.40883495$ \\
  $P_4$-lite  & $2.90146936$ & $0.34465296$ & $0.40833239$ \\
  $P_4$-extendible  & $2.88492066$ & $0.34662998$ & $
0.40351731$ \\
  $P_4$-sparse & $2.72743550$ & $0.36664478$ & $0.37405701$ \\
  $P_4$-reducible & $2.71715531$ & $0.36803196$ & $0.37115484$ \\
  $P_4$-free & $\frac{1}{2\ln(2)-1}\approx 2.58869945$ & $2\ln(2)-1\approx 0.38629436$ & $0.35065840$ \\
  \hline
\end{tabular}
\end{center}

\section{Enumeration of graphs with a given induced subgraph}\label{sec5}

\subsection{Induced subtrees and subgraphs}
We recall that the size of a graph is its number of vertices, and the size of a tree is its number of leaves.

\begin{dfn}[Induced subgraph]\label{d12}
Let $G$ be a graph, $k$ a positive integer and $\mathfrak{I}$ a partial injection from the set of labels of $G$ to $\N$. The \emph{labeled subgraph $G_\mathfrak{I}$} of $G$ induced by $\mathfrak{I}$ is defined as:
\begin{itemize}
    \item The vertices of $G_\mathfrak{I}$ are the vertices of $G$ whose label $\ell$ is in the domain of $\mathfrak{I}$. For every such vertex, we replace the label $\ell$ of the vertex by $\mathfrak{I}(\ell)$;
    \item For two vertices $v$ and $v'$ of $G_\mathfrak{I}$, $(v,v')$ is an edge of $G_\mathfrak{I}$ if and only if it is an edge of $G$.
\end{itemize}
\end{dfn}

\begin{dfn}[First common ancestor]
Let $t$ be a rooted tree and let $\ell_1,\ell_2$ be two distinct leaves of $t$. The \emph{first common ancestor} of $\ell_1$ and $\ell_2$ is the internal node of $t$ that is the furthest from the root and that belongs to the shortest path from the root to $\ell_1$, and the shortest path from the root to $\ell_2$.
\end{dfn}

\begin{dfn}[Induced subtree]\label{defg}
Let $(t,\mathfrak{I})$ be a marked tree in $\mathcal{T}_0$ ($\mathcal{T}_0$ is defined in \cref{d3}, and the notion of marked tree in \cref{mark}). The \emph{induced subtree $t_\mathfrak{I}$} of $t$ induced by $\mathfrak{I}$ is defined as:
\begin{itemize}
\item The leaves of $t_\mathfrak{I}$ are the leaves of $t$ that are marked. For every such leaf labeled with an integer $\ell$, the new label of $\ell$ is $\mathfrak{I}(\ell)$;

\item The internal nodes of $t_\mathfrak{I}$ are the internal nodes of $t$ that are first common ancestors of two or more leaves of $t_\mathfrak{I}$;

\item The ancestor-descendent relation in $t_\mathfrak{I}$ is inherited from the one in $t$;

\item For every internal node $v$ of $t$ that appears in $t_\mathfrak{I}$, let $H$ be its decoration in $t$. Denote by $J$ the set of positive integers $k$ such that the $k$-th tree of $t_v$ contains a leaf of $t_\mathfrak{I}$. For every $k$ in $J$, we define $\mathfrak{L}(k)$ as the smallest image by $\mathfrak{I}$ of a marked leaf label in the $k$-th tree of $t_v$. The decoration of $v$ in $t_\mathfrak{I}$ is the reduction of $H_{\mathfrak{L}}$.

\end{itemize}

For every internal node $v$ (resp.~leaf $\ell)$ of $t_\mathfrak{I}$, we also define $\phi(v)$ to be the only internal node (resp.~leaf) of $t$ corresponding to $v$.
\end{dfn}

\begin{rem}
When $(t,\mathfrak{I})$ is a marked tree and $t'$ is a subtree of $t$, we will denote $t'_\mathfrak{I}$ the tree induced by the restriction of $\mathfrak{I}$ to the set of labels of leaves of $t'$.
\end{rem}

As a consequence of Definitions \ref{d12} and \ref{defg}, we obtain:

\begin{lem}\label{çasevoit}
Let $(t,\mathfrak{I})$ be a marked tree in $\mathcal{T}_0$. Then $$\mathrm{Graph}(t)_{\mathfrak{I}}=\mathrm{Graph}(t_{\mathfrak{I}}).$$
\end{lem}

\begin{figure}[htbp]
\begin{center}
\centering
\includegraphics[scale=0.8]{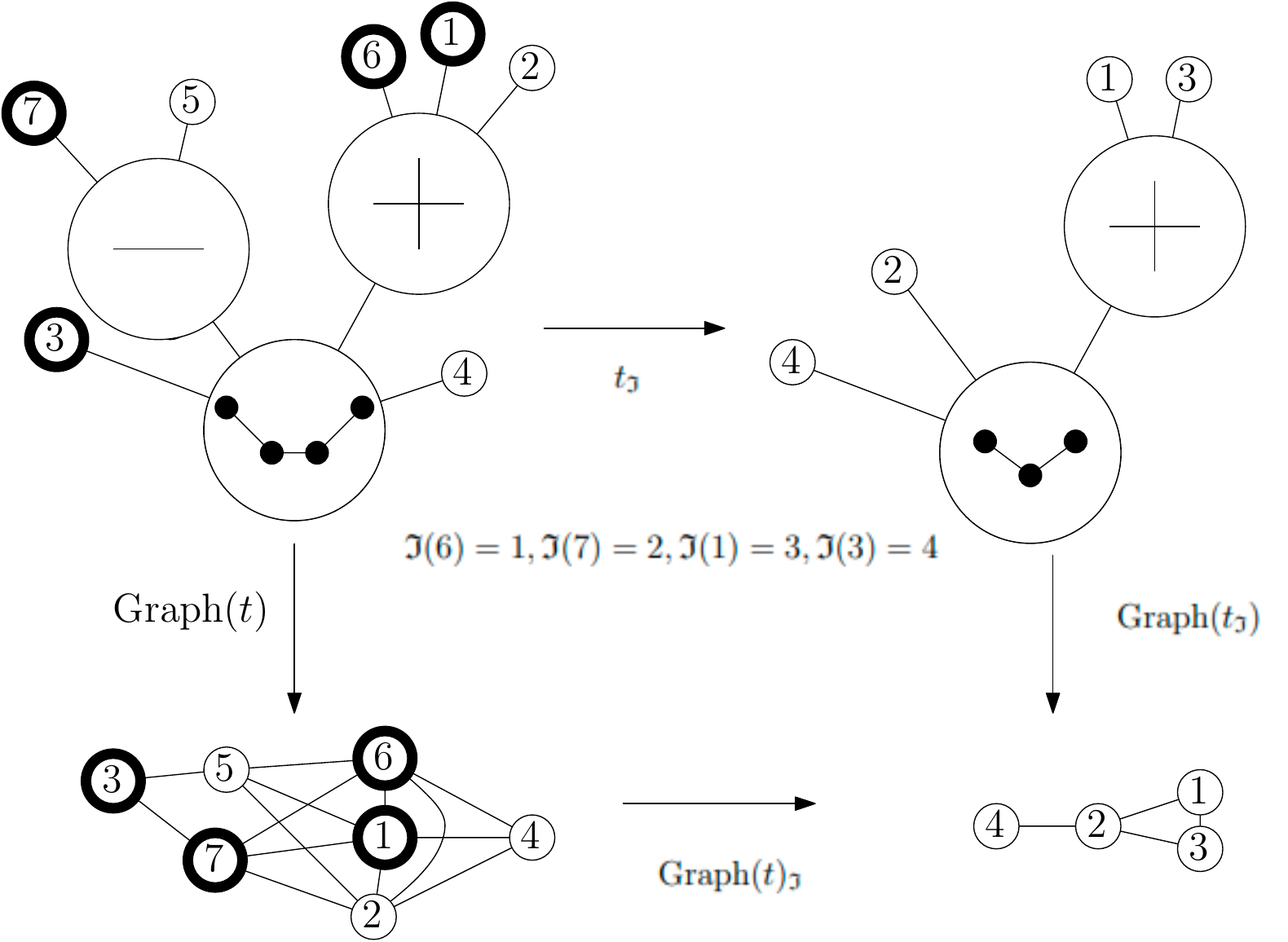}
\caption{Relations between induced subgraph and induced subtree.}
\end{center}
\end{figure}

\begin{dfn}
For every pair of graphs $(G,H)$ such that $G$ has no blossom and $H$ has at most one blossom, let \emph{$\mathrm{Occ}_{G}(H)$} be the number of partial injection $\mathfrak{I}$ from the vertex labels of $H$ to $\N$ such that no blossom is marked and $H_\mathfrak{I}$ is isomorphic to $G$.

\end{dfn}

\begin{dfn}
For every pair of graphs $(G,H)$ and $a\in \N$ such that $G$ has no blossom, $H$ has exactly one blossom and $a$ is the label of a vertex of $G$, let \emph{$\mathrm{Occ}_{G,a}(H)$} be the number of partial injection $\mathfrak{I}$ from the vertex labels of $G$ to $\N$ such that the image of the blossom by $\mathfrak{I}$ is $a$ and $H_\mathfrak{I}$ is isomorphic to $G$.

\end{dfn}

\begin{figure}[htbp]
\begin{center}
\centering
\includegraphics[scale=0.6]{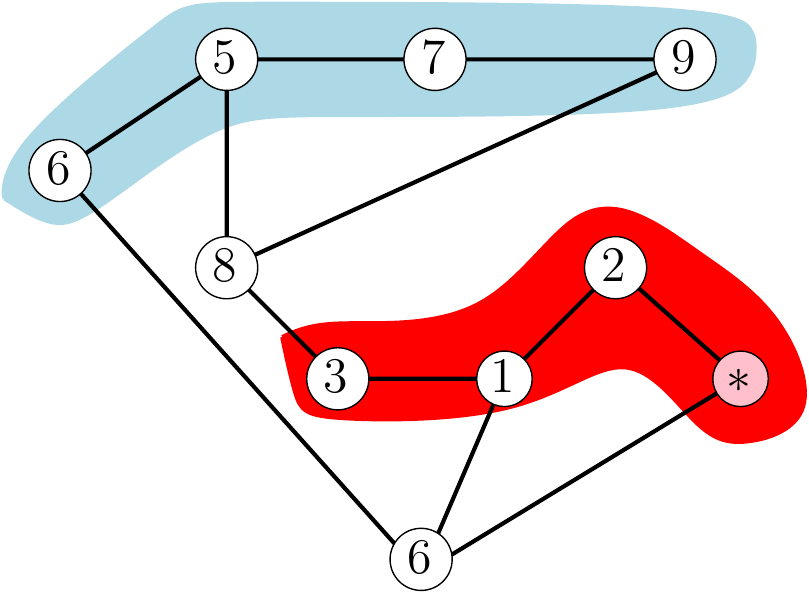}
\caption{Two occurences of a $P_4$ in a blossomed graph $H$. If $G$ is a $P_4$, the blue one is counted twice in $\mathrm{Occ}_{G}(H)$, the red one in counted once in $\mathrm{Occ}_{G,a}(H)$ iff $a$ is the label of an extremity of $G$.}
\end{center}
\end{figure}

\begin{dfn}
For every graph $G$ without blossom, and every $a\in\{1,\dots, N(G)=|G|\}$, set:
\begin{align*}
&\mathrm{Occ}_{G, \mathcal{P}}(z):=\sum\limits_{H\in \mathcal{P}}\frac{\mathrm{Occ}_{G}(H)z^{N(H)-N(G)}}{N(H)!};\qquad  \mathrm{Occ}_{G, \mathcal{P}^{\bullet}}(z):=\sum\limits_{H\in \mathcal{P}^{\bullet}}\frac{\mathrm{Occ}_{G}(H)z^{N(H)-N(G)}}{N(H)!}\\
&\mathrm{Occ}_{G,a, \mathcal{P}^{\bullet}}(z):=\sum\limits_{H\in \mathcal{P}^{\bullet}}\frac{\mathrm{Occ}_{G,a}(H)z^{N(H)-N(G)+1}}{N(H)!}
\end{align*}

\end{dfn}

\begin{nota}
$\mathrm{Occ}_{G,\dots}$ will only be used for graphs $G$ with no blossom.
\end{nota}

\begin{prop}
For every $k\geq 1$ and every $a\in\{1,\dots, k\}$:
\begin{align}
  \label{e1}
&\sum\limits_{G: \ N(G)=k}\mathrm{Occ}_{G,\mathcal{P}}(z)=P^{(k)}(z)\\
\label{e2}
&\sum\limits_{G: \ N(G)=k}\mathrm{Occ}_{G,\mathcal{P}^{\bullet}}(z)=(P^{\bullet})^{(k)}(z)\\
\label{e3}
&\sum\limits_{G:\ N(G)=k}\mathrm{Occ}_{G,a,\mathcal{P}^{\bullet}}(z)=(P^{\bullet})^{(k-1)}(z)
\end{align}

Thus for every graph $G$ with no blossom and every $a\in \{1,\dots,N(G)\}$, $\mathrm{Occ}_{G,\mathcal{P}}$, $\mathrm{Occ}_{G,\mathcal{P}^{\bullet}}$ and $\mathrm{Occ}_{G,a,\mathcal{P}^{\bullet}}$ have a radius of convergence strictly greater than $R$, the radius of convergence of $T$.
\end{prop}

\begin{proof}
Let $H$ be an element of $\mathcal{P}$. Since there are $\frac{N(H)!}{(N(H)-k)!}$ choices of partial injection whose image is $\{1,\dots, k\}$, we have:
$$\sum\limits_{G:\ N(G)=k}\mathrm{Occ}_{G,\mathcal{P}}(z)=\sum \limits_{H\in \mathcal{P}}\sum\limits_{G: \ N(G)=k}\frac{\mathrm{Occ}_{G}(H)z^{N(H)-k}}{N(H)!}=\sum \limits_{H\in \mathcal{P}}\frac{z^{N(H)-k}}{(N(H)-k)!}=P^{(k)}(z)$$

The proofs of \cref{e2,e3} are similar. In \cref{e3}, since $\mathfrak{I}^{-1}(a)$ must be $*$, there are exactly $\frac{N(H)!}{(N(H)-(k-1))!}$ choices for the partial injection.

For every graph $G$, $\mathrm{Occ}_{G,\mathcal{P}}$ has non-negative coefficients and for every $k\geq 0$, as mentioned in \Cref{sec4.2}, $P^{(k)}$ has a radius of convergence at least $R_0$, the minimum of the radii of convergence $P$ and $P^{\bullet}$, which is greater than $R$. This implies that $\mathrm{Occ}_{G,\mathcal{P}}$ has a radius of convergence greater than $R$. The proof for the other series is similar.\end{proof}

\subsection{Enumerations of trees with a given induced subtree}

The key step in the proof of our main theorem is to compute the limiting probability (when $n\to \infty$) that a uniform induced subtree of a uniform tree in $\mathcal{T}_{\mathcal{P},\mathcal{P}^{\bullet}}$ with $n$ leaves is a given substitution tree. 

In the following, let $\tau\in \mathcal{T}_0$ be a fixed substitution tree of size at least $2$.

\begin{dfn}
We define \emph{$\mathcal{T}_{\tau}$} to be the set of marked trees $(t,\mathfrak{I})$ where $t\in \mathcal{T}_{\mathcal{P},\mathcal{P}^{\bullet}}$ and $\mathfrak{I}$ is such that $t_\mathfrak{I}$ is isomorphic to $\tau$. We also define $T_{\tau}$ to be the corresponding exponential generating function (where the size parameter is the total number of leaves, including the marked ones).
\end{dfn}
 The aim now is to decompose a tree admitting $\tau$ as a subtree in smaller trees. Let $(t,\mathfrak{I})$ be in $\mathcal{T}_{\tau}$. A prime node $v$ of $\tau$ is such that $t[\phi(v)]$ is either in case $(D2)$ or $(D4)$ of \cref{d1}: in other word, $\phi(v)$ must be a prime node. In constrast, knowing that an internal node $v'$ of $\tau$ is decorated with $\oplus$ or $\ominus$ does not give any information about the decoration of $\phi(v')$.

In order to state \cref{thmcoeur} below, we need to partition the internal nodes of $\tau$:

\begin{dfn}
Let $(t,\mathfrak{I})$ be in $\mathcal{T}_{\tau}$. We denote by $\mathbf{V}(t,\mathfrak{I})$ the set of internal nodes $v$ of $\tau$ such that $\phi(v)$ is non-linear. The set $\mathbf{V}(t,\mathfrak{I})$ can be partitioned in $4$ subsets: 
\begin{itemize}
    \item $\mathbf{V_0}(t,\mathfrak{I})$ the set of internal nodes $v$ such that $t[\phi(v)]$ is in case $(D2)$;
    \item $\mathbf{V_{1}}(t,\mathfrak{I})$ the set of internal nodes $v$ such that $t[\phi(v)]$ is in case $(D4)$ and no marked leaf is in the $i$-th tree of $t_{\phi(v)}$  (where $i$ is the element such that $(D4)$ holds in \Cref{d1});
    \item $\mathbf{V_{2}}(t,\mathfrak{I})$ the set of internal nodes $v$ such that $t[\phi(v)]$ is in case $(D4)$ and exactly one marked leaf is in the $i$-th tree of $t_{\phi(v)}$  (where $i$ is the element such that $(D4)$ holds in \Cref{d1});
    \item $\mathbf{V_{3}}(t,\mathfrak{I})$ the set of internal nodes $v$ such that $t[\phi(v)]$ is in case $(D4)$ and at least two marked leaves are in the $i$-th tree of $t_{\phi(v)}$  (where $i$ is the element such that $(D4)$ holds in \Cref{d1}).
    
\end{itemize}
\end{dfn}

Note that the set of non-linear nodes of $\tau$ must be included in $\mathbf{V}(t,\mathfrak{I})$. Since for every element $v$ of $\mathbf{V}(t,\mathfrak{I})$ at most one element of $t_{\phi(v)}$ is non trivial, at most one element of $\tau_v$ is non trivial. Thus if $\tau$ has some non-linear nodes $v$ such that two or more elements of $\tau_v$ are not reduced to a single leaf, $\mathcal{T}_{\tau}=\emptyset$. In the following, we assume that it is not the case for $\tau$. If $\tau_v$ has exactly one non trivial tree, then $v\in \mathbf{V_{3}}(t,\mathfrak{I})$. Otherwise, $\tau_v$ is a union of leaves.

\begin{nota}
We denote by $U_0$ (resp.~$U_1$) the set of internal nodes $v$ of $\tau$ such that no tree (resp.~exactly one tree) of $\tau_v$ has size greater or equal to $2$. 

\end{nota}

Note that by definition $\mathbf{V_0}(t,\mathfrak{I})\cup \mathbf{V_1}(t,\mathfrak{I})\cup \mathbf{V_2}(t,\mathfrak{I})\subset U_0$ and $\mathbf{V_3}(t,\mathfrak{I})\subset U_1$.

We also define $\mathbf{rk}_{t,\mathfrak{I}}: \mathbf{V_{2}}(t,\mathfrak{I})\mapsto \N$ as follows. Let $v\in \mathbf{V_{2}}(t,\mathfrak{I})$, we define $\mathbf{rk}_{t,\mathfrak{I}}(v)$ to be the only integer $k$ such that, if $\ell$ is the label of the $k$-th leaf of $\tau_v$ then the leaf of label $\mathfrak{I}^{-1}(\ell)$ in $t$ belongs to the $i$-th tree of $t_{\phi(v)}$  (where $i$ is the element such that $(D4)$ holds in \Cref{d1}). For every $v\in \mathbf{V_{2}}(t,\mathfrak{I})$, we have $1\leq \mathbf{rk}_{t,\mathfrak{I}}(v)\leq |\tau_v|$.

\begin{thm}\label{thmcoeur}
Let $\tau$ be a substitution tree of size at least $2$ such that every non-linear node of $\tau$ is in $U_0\cup U_1$. Let $V_0, V_{1}$ and $V_{2}$ be three disjoint subsets of $U_0$ and let $V_{3}$ be a subset of $U_1$ such that every non-linear node of $\tau$ is in $\Vs:= V_0\cup V_{1}\cup  V_{2}\cup V_{3}$. Let $\mathrm{rk}$: $V_{2}\to \N$ be such that $1\leq \mathrm{rk}(w)\leq |\tau_w|$ for every $w\in V_{2}$. 

Let $\mathcal{T}_{\tau,V_0, V_{1}, V_{2}, V_{3},\mathrm{rk}}$ be the set of marked trees $(t,\mathfrak{I})$ in $\mathcal{T}_{\tau}$ such that $\mathbf{V_0}(t,\mathfrak{I})=V_0,\mathbf{V_{1}}(t,\mathfrak{I})=V_{1},\mathbf{V_{2}}(t,\mathfrak{I})=V_{2}, \mathbf{V_{3}}(t,\mathfrak{I})=V_{3}, \mathbf{rk}_{t,\mathfrak{I}}=\mathrm{rk}$, and let $T_{\tau,V_0,V_{1}, V_{2}, V_{3},\mathrm{rk}}$ be its exponential generating function.

Then
\begin{align*}
T_{\tau,V_0,V_{1}, V_{2}, V_{3},\mathrm{rk}}=z^{|\tau|}&T^{\mathrm{root}}\left(T_{\mathrm{not}\oplus}^{\oplus}\right)^{d_{=}}\left(T_{\mathrm{not}\oplus}^{\ominus}\right)^{d_{\neq }}\left(T_{\mathrm{not}\oplus}^{\mathsf{blo}}\right)^{d_{\overline{\Vs}\to\Vs}}\left(T_{\mathrm{not}\oplus}^{'}\right)^{d_{\overline{\Vs}\to \ell}}\exp(n_L T_{\mathrm{not}\oplus})\\
&\times T^{|V_{1}|}T'^{|V_{2}|}(T^{\oplus})^{n_1}(T^{\mathsf{blo}})^{n_2}F
\end{align*}
where 
$$F:=\prod\limits_{v\in V_0}\mathrm{Occ}_{\mathrm{dec}(v),\mathcal{P}}\prod\limits_{v\in V_{3}}\mathrm{Occ}_{\mathrm{dec}(v),\mathrm{br}(v),\mathcal{P}^{\bullet}}\prod\limits_{v\in V_{1}}\mathrm{Occ}_{\mathrm{dec}(v),\mathcal{P}^{\bullet}}\prod\limits_{v\in V_{2}}\mathrm{Occ}_{\mathrm{dec}(v),\mathrm{rk}(v),\mathcal{P}^{\bullet}}$$
and:
\begin{itemize}
    \item $d_{=}$ is the number of edges between two internal nodes not in \Vsr with the same decoration ($\oplus$ and $\oplus$, or $\ominus$ and $\ominus$);
    \item $d_{\neq }$ is the number of edges between two internal nodes not in \Vsr decorated with different decorations ($\oplus$ and $\ominus$);
    \item $d_{\overline{\Vs}\to\Vs}$ is the number of edges between an internal node not belonging to \Vsr and one of its children belonging to \Vsr;
    \item $d_{\overline{\Vs}\to \ell}$ is the number of edges between an internal node not in \Vsr and a leaf;
    \item $n_L$ is the number of internal nodes not in \Vsr;
    \item $\mathrm{dec}(v)$ is the decoration of $v$;
    \item for every $v\in V_3$, $\mathrm{br}(v)$ is the position of the tree of $\tau_v$ not reduced to a leaf;
    \item $n_1$ (resp.~$n_2$) is the number of internal nodes $v$ in $V_{3}$ such that the root of the $\mathrm{br}(v)$-th tree of $\tau_v$ is not in \Vsr (resp.~is in \Vsr);
    \item $T^{\mathrm{root}}=T^{\oplus}$ if the root of $\tau$ is not in \Vsr, $T^{\mathrm{root}}=T^{\mathsf{blo}}$ otherwise.
\end{itemize}

\end{thm}

\begin{figure}
\begin{center}
\centering
\includegraphics[width=15cm]{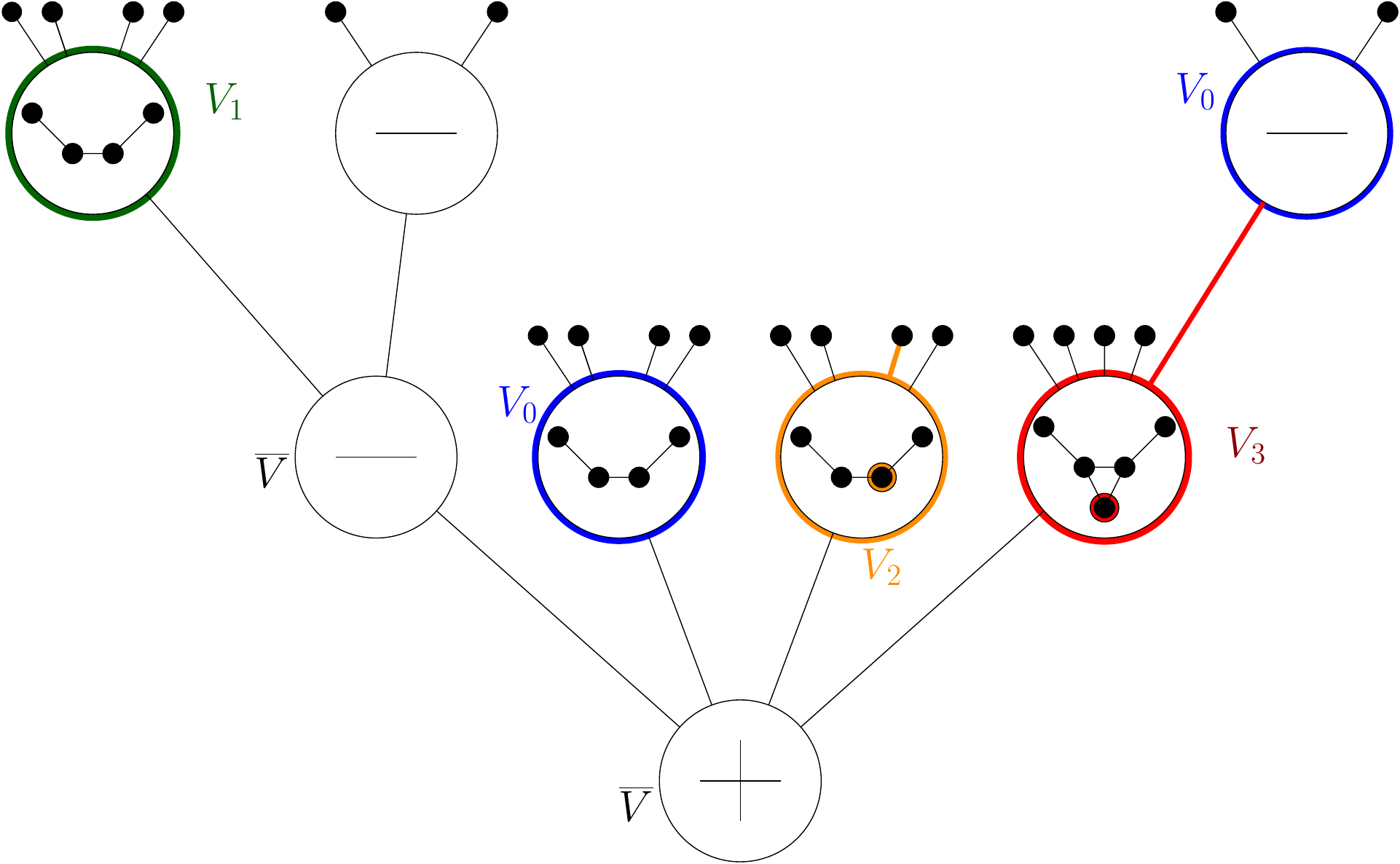}
\caption{A substitution tree $\tau$ and corresponding sets $V_0,V_1,V_2,V_3$}\label{fig13}
\end{center}
\end{figure}

\begin{figure}[htbp]
\begin{center}
\centering
\includegraphics[width=16cm]{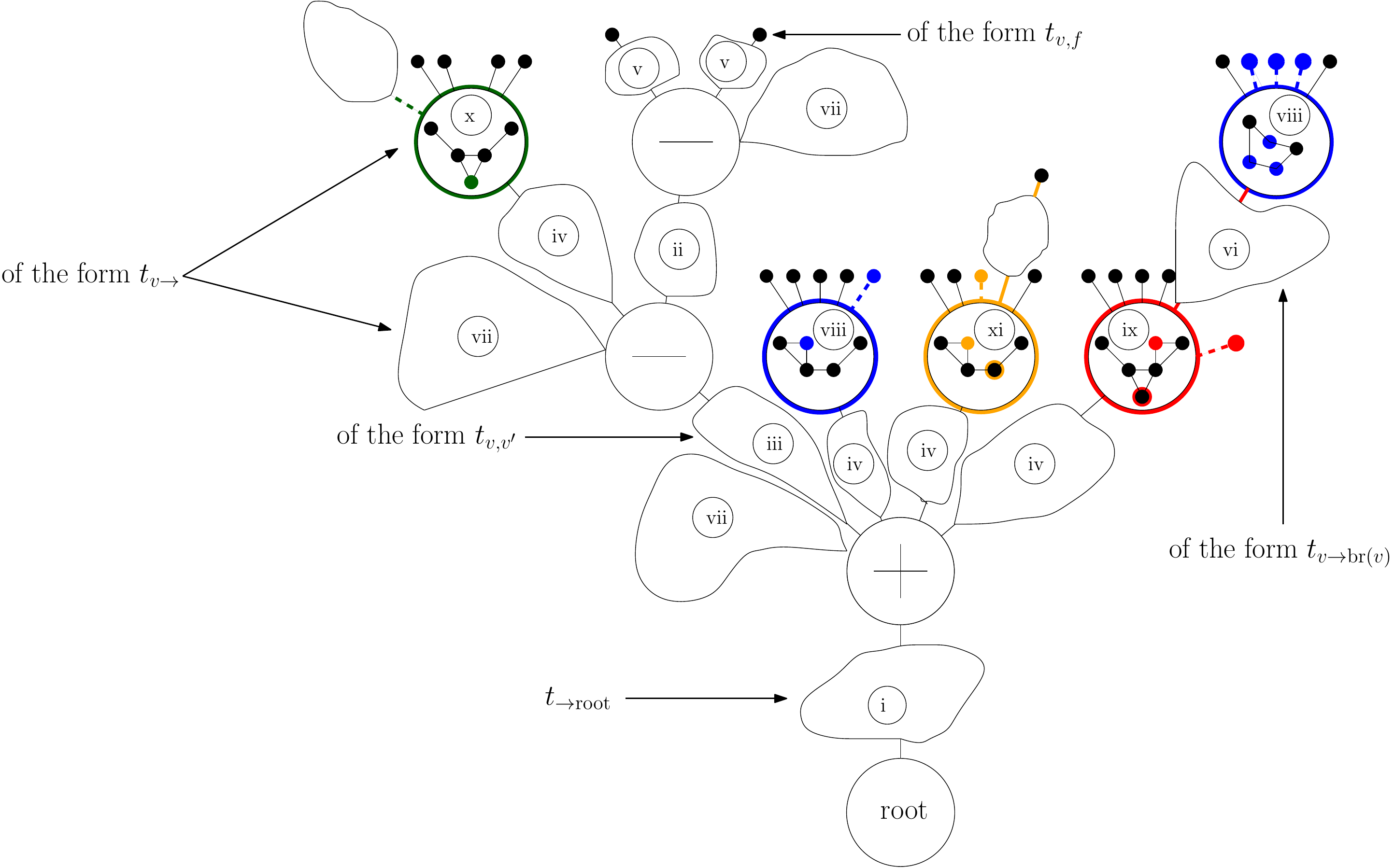}
\caption{The decomposition of a tree admitting the graph $\tau$ of \cref{fig13} as an induced tree. Roman numerals correspond to the different cases of the proof of \Cref{thmcoeur}.}
\end{center}
\end{figure}

\begin{proof}
Let $t$ be a tree in $\mathcal{T}_{\tau,V_0,V_{1}, V_{2}, V_{3},\mathrm{rk}}$.
We decompose $t$ into several disjoints subtrees. The blossoms are nodes where (the root of) an other tree will be glued (and thus they are not counted in the generating series, to avoid counting them twice). In the following, every defined tree will be assumed to be reduced.

\noindent We define $t_{\to\mathrm{root}}$ to be the tree $t$ blossomed at $\phi(r_0)$, where $r_0$ is the root of $\tau$.

\noindent For each $v\in V_{\tau}$, we define the tree $t_{v\to}$ in the following way: 

\begin{itemize}
    \item  If $v$ is not in \Vsr,  $t_{v\to}$ is the subtree of $t$ containing $\phi(v)$ and all the trees of $t_{\phi(v)}$ that do not contain a marked leaf of $t$.
    \item If $v$ is in $V_0\cup V_{1}\cup V_{2}$, $t_{v\to}$ is the tree $t[\phi(v)]$. 
    \item If $v$ is in $V_{3}$, $t_{v\to}$ is the tree $t[\phi(v)]$ obtained after blossoming the root of the non trivial tree of $t_{\phi(v)}$. The blossom is marked with the smallest mark in the non trivial tree of $t_{\phi(v)}$.
\end{itemize}

\noindent For every internal nodes $v,v'$ in $\tau$ such that $v$ is not in \Vsr and $v'$ is a child of $v$, let $t_{v\to v'}$ be the unique tree of $t_{\phi(v)}$ containing $\phi(v')$, blossomed at $\phi(v')$.

\noindent For every internal node $v$ in $\tau$ not in \Vsr, and every leaf $f$ which is a child of $v$ in $\tau$, we define $t_{v\to f}$ to be the tree of $t_{\phi(v)}$ containing $\phi(f)$.

\noindent For every internal node $v$ in $V_{3}$, we define $t_{v\to \mathrm{br}(v)}$ to be the non trivial tree of $t_{\phi(v)}$ blossomed at $\phi(v')$, where $v'$ is the root of the $\mathrm{br}(v)$-th tree of $\tau_v$.

Now we need to analyze the properties of the trees that appear in this decomposition and compute the corresponding exponential generating function. In the rest of the proof, we will say abusively that every blossomed tree belongs to $\mathcal{T}_{\mathcal{P},\mathcal{P}^{\bullet}}$, and that two nodes both decorated with $\oplus$ or $\ominus$ have the same decoration, even if they do not have the same number of children.

\noindent{\bf (i): analysis of $t_{\to\mathrm{root}}$}

The tree $t_{\to\mathrm{root}}$ is a tree in $\mathcal{T}_{\mathcal{P},\mathcal{P}^{\bullet}}$, it has no marked leaf and a unique blossom. If the root is not in \Vsr and decorated with $\oplus$ (resp.~$\ominus$), the blossom is $\oplus$-replaceable (see \Cref{dfnblossom}) (resp.~$\ominus$-replaceable). If the root is in \Vsr, the blossom is replaceable.

The corresponding exponential generating function is equal to $T^{\oplus}$ if the root is not in \Vsr and equal to $T^{\mathsf{blo}}$ otherwise.

\noindent{\bf (ii): analysis of $t_{v\to v'}$ where $v\not\in \Vs$ and $v'$ is a child of $v$ not in \Vsr with the same decoration}

The tree $t_{v\to v'}$ is a tree in $\mathcal{T}_{\mathcal{P},\mathcal{P}^{\bullet}}$ whose root is not decorated with the same decoration as $v$ and with one blossom $\oplus$-replaceable if $v'$ is decorated with $\oplus$, $\ominus$-replaceable otherwise and no marked leaf. 

The exponential generating function of such trees is either $T_{\mathrm{not}\oplus}^{\oplus}$ if both nodes are decorated with $\oplus$ or $T_{\mathrm{not}\ominus}^{\ominus}$ if both nodes are decorated with $\ominus$, which are both equal.

\noindent{\bf (iii): analysis of $t_{v\to v'}$ where $v\not\in \Vs$ and $v'$ is a child of $v$ not in \Vsr with a different decoration}

The tree $t_{v\to v'}$ is a tree in $\mathcal{T}_{\mathcal{P},\mathcal{P}^{\bullet}}$ whose root is not decorated with the same decoration as $v$ and with one blossom $\oplus$-replaceable if $v'$ is decorated with $\oplus$, $\ominus$-replaceable otherwise and no marked leaf.

The exponential generating function of such trees is either $T_{\mathrm{not}\oplus}^{\ominus}$ if $v$ is decorated with $\oplus$ and $v'$ with $\ominus$ or $T_{\mathrm{not}\ominus}^{\oplus}$ if $v$ is decorated with $\ominus$ and $v'$ with $\oplus$, which are both equal.

\noindent{\bf (iv): analysis of $t_{v \to v'}$ where $v\not\in \Vs$ and $v'$ is a child of $v$ in \Vsr}

The tree $t_{v\to v'}$ is a tree in $\mathcal{T}_{\mathcal{P},\mathcal{P}^{\bullet}}$ whose root is not decorated with the decoration of $v$ with one blossom and no marked leaf. 

The corresponding exponential generating function is $T_{\mathrm{not}\oplus}^{\mathsf{blo}}$.

\noindent{\bf (v): analysis of $t_{v\to f}$ where $v\not\in \Vs$ and $f$ is a leaf which is a child of $v$}

The tree $t_{v\to f}$ is a tree in $\mathcal{T}_{\mathcal{P},\mathcal{P}^{\bullet}}$ whose root is not decorated with the decoration of $v$ with one marked leaf and no blossom.

The corresponding exponential generating function is $zT_{\mathrm{not}\oplus}'$.

\noindent{\bf (vi): analysis of $t_{v\to \mathrm{br}(v)}$ where $v\in V_{3}$ }

The tree $t_{v\to \mathrm{br}(v)}$ is a tree with a blossom that is replaceable if the root of the $\mathrm{br}(v)$-th subtree of $t[v]$ is in \Vsr, $\oplus$-replaceable (resp.~$\ominus$-replaceable) if the root is not in \Vsr and labeled $\oplus$ (resp.~$\ominus$), with no marked leaf.

The corresponding exponential generating function is equal to $T^{\oplus}$ if the root of the $\mathrm{br}(v)$-th tree of $\tau_v$ is not in \Vsr and equal to $T^{\mathsf{blo}}$ otherwise.

\noindent{\bf (vii): analysis of $t_{v\to }$ where $v\not\in \Vs$ }

The tree $t_{v\to}$ is a tree whose root denoted is decorated with the same decoration as $v$, who has no marked leaf and no blossom. It verifies all the conditions of being $(\mathcal{P},\mathcal{P}^{\bullet})$-consistent, except that the root can have $0$ or $1$ child.

The corresponding exponential generating function is $\sum\limits_{k\geq 0}T_{\mathrm{not}\oplus}^k=\exp(T_{\mathrm{not}\oplus})$.

\noindent{\bf (viii): analysis of $t_{v\to}$ where $v\in V_0$ }

The tree $t_{v\to}$ is a tree in $\mathcal{T}_{\mathcal{P},\mathcal{P}^{\bullet}}$ whose root is decorated with an element of $\mathcal{P}$. The subtree induced by the marked leaves of $t_{v\to}$ is $\tau[v]$. Moreover $t_{v\to}$ has only one internal node.

 The corresponding exponential generating function is $$\sum\limits_{H\in \mathcal{P}}\frac{\mathrm{Occ}_{\mathrm{dec}(v)}(H)z^{N(H)}}{N(H)!}=z^{N(\mathrm{dec}(v))}\mathrm{Occ}_{\mathrm{dec}(v),\mathcal{P}}.$$
Indeed, for a given $H\in\mathcal{P}$, the term $\frac{z^{N(H)}}{N(H)!}$ correspond to the set of leaves and the term $\mathrm{Occ}_{\mathrm{dec}(v)}(H)$ to the possible markings.

\noindent{\bf (ix): analysis of $t_{v\to}$ where $v\in V_{3}$ }

The tree $t_{v\to}$ is a tree $(\mathcal{P},\mathcal{P}^{\bullet})$-consistent in case $(D4)$ of \cref{d1}. The subtree induced by the marked leaves of $t_{v\to}$ is $\tau[v]$, where the non-trivial tree of $\tau_v$ is replaced by a blossom, marked with the smallest mark in the non-trivial tree of $\tau_v$. Moreover $t_{v\to}$ has only one internal node.

Similarly to case (viii), the corresponding exponential generating function is: $$\sum\limits_{H\in \mathcal{P}^{\bullet}}\frac{\mathrm{Occ}_{\mathrm{dec}(v), \mathrm{br}(v)}(H)z^{N(H)}}{N(H)!}=z^{N(\mathrm{dec}(v))-1}\mathrm{Occ}_{\mathrm{dec}(v),\mathrm{rk}(v),\mathcal{P}^{\bullet}}.$$

\noindent{\bf (x): analysis of $t_{v\to}$ where $v\in V_{1}$ }

The tree $t_{v\to}$ is a tree $(\mathcal{P},\mathcal{P}^{\bullet})$-consistent in case $(D4)$ of \cref{d1}. The subtree induced by the marked leaves of $t_{v\to}$ is $\tau[v]$ and no marked leaf belongs to the $i$-th tree of $t_{\phi(v)}$ (where $i$ is the element such that $(D4)$ holds in \Cref{d1}).

The corresponding exponential generating function is: $$\sum\limits_{H\in \mathcal{P}^{\bullet}}\frac{\mathrm{Occ}_{\mathrm{dec}(v)}(H)z^{N(H)}}{N(H)!}\times T=z^{N(\mathrm{dec}(v))}\mathrm{Occ}_{\mathrm{dec}(v),\mathcal{P}^{\bullet}}\times T.$$

The sum corresponds to the choice of the root (as in the previous cases), and the factor $T$ to the potential non trivial tree of $t_v$.

\noindent{\bf (xi): analysis of $t_{v\to}$ where $v\in V_{2}$ }

The tree $t_{v\to}$ is a tree $(\mathcal{P},\mathcal{P}^{\bullet})$-consistent in case $(D4)$ of \cref{d1}. The subtree induced by the marked leaves of $t_{v\to}$ is $\tau[v]$ and there is only one marked leaf $\ell$ in the $i$-th tree of $t_{\phi(v)}$ (where $i$ is the element such that $(D4)$ holds in \Cref{d1}).
Moreover, if we denote by $j$ the label of $\ell$, the label of the $\mathrm{rk}(v)$-th leaf of $\tau_v$ is $\mathfrak{I}(j)$.

Similarly to case (x), the corresponding exponential generating function is:
$$\sum\limits_{H\in \mathcal{P}^{\bullet}}\frac{\mathrm{Occ}_{\mathrm{dec}(v), \mathrm{rk}(v)}(H)z^{N(H)}}{N(H)!}\times zT'=z^{N(\mathrm{dec}(v))}\mathrm{Occ}_{\mathrm{dec}(v),\mathrm{rk}(v),\mathcal{P}^{\bullet}}\times T'.$$\medskip

All these conditions ensure that we can recover $t$ by gluing all the different trees and that the subtree of $t$ induced by $\mathfrak{I}$ is $\tau$. Thus, $T_{\tau,V_0,V_1,V_2,V_3,\mathrm{rk}}$ is the product of the generating functions and this concludes the proof of the theorem.

\end{proof}

Using \cref{ecut,ecus,eqblo,ecuua,ecuub,ecuuc,ecuud} and Singular Differentiation \cite[Theorem VI.8 p.419]{flajolet2009analytic} we get the following corollary:
\begin{cor}
The series $T_{\tau,V_0,V_1,V_2,V_3,\mathrm{rk}}$ has radius at least $R$, is $\Delta$-analytic and its asymptotic expansion near $R$ is: 
$$T_{\tau,V_0,V_1,V_2,V_3,\mathrm{rk}}=C_{\tau,V_0,V_1,V_2,V_3,\mathrm{rk}}\left(1-\frac{z}{R}\right)^{\beta}\left(1+o(1)\right)$$
where
$$C_{\tau,V_0,V_1,V_2,V_3,\mathrm{rk}}:=\alpha  \kappa^{\gamma}(1+P^{\bullet}(R))^{\theta}(1-P^{\bullet}(R))^{|V_{1}|} 2^{\lambda}R^{\mu}\times F(R)$$

with 
\begin{align*}
\beta&=-\frac{1+d_=+d_{\neq}+d_{\overline{\Vs}\to\Vs}+d_{\overline{\Vs}\to \ell}+|V_{2}|+|V_{3}|}{2}\\
    \gamma&=d_{\overline{\Vs}\to \ell}+|V_{2}|-d_=-d_{\neq}-d_{\overline{\Vs}\to\Vs}-|V_{3}|-1\\
    \theta&=d_=+d_{\neq}-|V_{1}|-|V_{2}|-n_2-n_L\\
    \lambda &=-d_{\overline{\Vs}\to \ell}-n_1-2d_=-2d_{\neq}-d_{\overline{\Vs}\to\Vs}+n_L\\
    \mu &=-d_{\overline{\Vs}\to \ell}-|V_{2}|+|\tau|
\end{align*}

and $\alpha=\frac{1}{2}$ if the root is not in \Vsr, $\frac{1}{1+P^{\bullet}(\kappa)}$ otherwise.

\end{cor}

\section{Proof of the main theorems}\label{sec6}
\subsection{Background on graphons}

We now review the necessary material on graphons. We refer the reader to \cite{lovasz} for a comprehensive presentation of deterministic graphons, while \cite{janson} studies specifically the convergence of random graphs in the sense of graphons. Here we will only recall the properties needed to prove the convergence of random graphs toward the Brownian cographon (see \cite{bassino2021random}).

\begin{dfn}\label{Graphon}
A graphon is an equivalence class of symmetric functions $f:[0,1]^2\mapsto [0,1]$, under the equivalence relation $\sim$, where $f\sim g$ if there exists a measurable function $\phi:[0,1]\mapsto[0,1]$ that is invertible and measure preserving such that, for almost every $(x,y)\in[0,1]^2$, $f(\phi(x),\phi(y))=g(x,y)$. We denote by $\tilde{\mathcal{W}}$ the set of graphons.
\end{dfn}

Intuitively graphons can be seen as continuous analogous of graph adjacency matrices, where graphs are considered  up to relabeling (hence the quotient by $\sim$). There is a natural way to embed a finite graph into graphons:

\begin{dfn}\label{graphe-graphon}
Let $G$ be a (random) graph of size $n$. We define the (random) graphon $W_G$ to be the equivalence class of $w_G:[0,1]^2\mapsto[0,1]$ defined by:
\begin{align*}
    \forall(x,y)\in[0,1]^2\quad w_G(x,y)&:=1_{\lceil nx\rceil \text{connected to}\lceil ny\rceil}
\end{align*}
\end{dfn}

There exists a metric $\delta_{\square}$ on the set of graphons $\tilde{\mathcal{W}}$ such that $(\tilde{\mathcal{W}},\delta_{\square})$ is compact \cite[Chapter 8]{lovasz}, thus we can define for $\delta_{\square}$ the convergence in distribution of a random graphon. If $(\mathbf{G}^{(n)})_{n\geq 1}$ is a sequence of random graphs, there exists a simple criterion \cite[Theorem 3.1]{janson} characterizing the convergence in distribution of $(W_{\mathbf{G}^{(n)}})$ with respect to $\delta_{\square}$:

\begin{thm}[Rephrasing of \cite{janson}, Theorem $3.1$]
\label{Portmanteau}
	For any $n$, let $\mathbf{G}^{(n)}$ be a random graph of size $n$. Denote by $W_{\mathbf{G}^{(n)}}$ the random graphon associated to ${\mathbf{G}^{(n)}}$.
	The following assertions are equivalent:
	\begin{enumerate}
		\item [(a)] The sequence of random graphons $(W_{\mathbf{G}^{(n)}})_{n\geq 1}$ converges in distribution to some random graphon $\mathbf{W}$. 
		\item [(b)] The random infinite vector $\left(\frac{\mathrm{Occ}_{\mathbf{G}^{(n)}}(H)}{n(n-1)\dots(n-|H|+1)}\right)_{H\text{ finite graph}}$ converges in distribution in the product topology to some random infinite vector $(\mathbf{\Lambda}_H)_{H\text{ finite graph}}$. 

	\end{enumerate}

\end{thm}

For a finite graph $H$, the random variable $\mathbf{\Lambda}_H$ can be seen as the density of the pattern $H$ in the graphon $\mathbf{W}$:  the variables $(\mathbf{\Lambda}_H)_H$ play the roles of margins of $\mathbf{W}$ in the space of graphons. 

For $k\geq 1$ and $\mathbf{W}$ a random graphon, we denote by $\mathrm{Sample}_k(\mathbf{W})$ the unlabeled random graph built as follows:
$\mathrm{Sample}_k(W)$ has vertex set $\{v_1,v_2,\dots,v_k\}$ and,
letting $(X_1,\dots,X_k)$ be i.i.d.~uniform random variables in $[0,1]$,
we connect vertices $v_i$ and $v_j$ with probability $w(X_i,X_j)$
(these events being independent, conditionally on $(X_1,\cdots,X_k)$ and $\mathbf{W}$). The construction does not depend on the representation of the graphon.

With the notations of \Cref{Portmanteau}, we have for any finite graph $H$ $$\mathbf{\Lambda}_H=\mathbb{P}(\mathrm{Sample}_{|H|}(\mathbf{W})=H\ \vert\  \mathbf{W}).$$

The article \cite{bassino2021random} introduces a random graphon $\mathbf{W}^{1/2}$ called the Brownian cographon which can be explicitly constructed as a function of a realization of a Brownian excursion. Besides, \cite[Proposition 5]{bassino2021random} states that the distribution of the Brownian cographon is characterized\footnote{This characterization is strongly linked to the remarkable property that $k$ uniform leaves in the CRT induce a uniform binary tree with $k$ leaves, see again \cite[Section 4.2]{bassino2021random}.} by the fact that for every $k\geq 2$,  $\mathrm{Sample}_k(\mathbf{W}^{1/2})$ has the same law as the unlabeled version of $\mathrm{Graph}(\mathbf{b}_k)$ with $\mathbf{b}_k$ a uniform labeled binary tree with $k$ leaves and i.i.d.~uniform decorations in $\{\oplus, \ominus \}$.

A consequence of this characterization is a simple criterion for convergence to the Brownian cographon.

\begin{lem}[Rephrasing of \cite{bassino2021random} Lemma $4.4$]\label{111}
For every positive integer $n$, let $\mathbf{T}^{(n)}$ be a uniform random tree in $\mathcal{T}_{\mathcal{P},\mathcal{P}^{\bullet}}$ with $n$ vertices.
	For every positive integer $\ell$, $\mathbf{\mathfrak{I}_\ell}^{(n)}$ be a uniform partial injection from $\{1,\dots,n\}$ to $\N$ whose image is $\{1,\dots, \ell\}$ and independent of $\mathbf{T}^{(n)}$. Denote by $\mathbf{T}_{\mathbf{\mathfrak{I}_\ell}^{(n)}}^{(n)}$ the subtree induced by $\mathbf{\mathfrak{I}_\ell}^{(n)}$.

	Suppose that for every $\ell$ and for every binary tree $\tau$ with $\ell$ leaves, 
	\begin{equation}\mathbb{P}(\mathbf{T}_{\mathbf{\mathfrak{I}}^{(n)}}^{(n)}=\tau) \xrightarrow[n\to\infty]{} \frac{(\ell-1)!}{(2\ell-2)!}.\label{eq:factorisation}\end{equation}
	Then $W_{\mathrm{Graph}(\mathbf{T}^{(n)})}$ converges as a graphon to the Brownian cographon $\mathbf{W}^{1/2}$ of parameter $1/2$.
	
\end{lem}

\subsection{Conclusion of the proof of \Cref{cgbintro}}

\begin{prop}\label{c2}
Let $\tau$ be a binary tree with $\ell\geq 2$ leaves. The series $T_{\tau}$ has radius of convergence $R$, is $\Delta$-analytic and its asymptotic expansion near $R$ is: 
\begin{align}
    \label{ee1}
T_{\tau}=\frac{ \kappa}{(1+P^{\bullet}(R)) 2^{2\ell-2}}\left(1-\frac{z}{R}\right)^{-\frac{2\ell-1}{2}}\left(1+o(1)\right).
\end{align}
\end{prop}

\begin{proof}
As
$$T_{\tau}=\sum \limits_{\tau,V_0,V_1,V_2,V_3,\mathrm{rk}}T_{\tau,V_0,V_1,V_2,V_3,\mathrm{rk}},$$
the asymptotic expansions of the different series $T_{\tau,V_0,V_1,V_2,V_3,\mathrm{rk}}$ yield the $\Delta$-analyticity of $T_{\tau}$, its asymptotic expansion and its radius of convergence.

Note that $\beta\leq \frac{1+e}{2}$ where $e$ is the number of edge of $\tau$, with equality if and only if $V_0,V_{1},V_{2}$ and $V_{3}$ are all empty.

Therefore, only the series $T_{\tau,\emptyset,\emptyset,\emptyset,\emptyset,\mathrm{rk}}$ contributes to the leading term of the asymptotic expansion. In this case, $d_{\overline{\Vs}\to \ell}=\ell$, $d_=+d_{\neq}=\ell-2$ and $n_L=\ell-1$ which gives the announced expansion.
\end{proof}

\begin{thm}\label{110}

Let $\tau$ be a binary tree with $\ell\geq 2$ leaves. For $n\geq \ell$ and $\mathbf{T}^{(n)}$ be a uniform random tree in $\mathcal{T}_{\mathcal{P},\mathcal{P}^{\bullet}}$ with $n$ vertices.
Let $\mathbf{\mathfrak{I}_{\ell}}^{(n)}$ be a uniform partial injection from $\{1,\dots,n\}$ to $\N$ whose image is $\{1,\dots, \ell\}$ and independent of $\mathbf{T}^{(n)}$. Denote by $\mathbf{T}_{\mathbf{\mathfrak{I}_{\ell}}^{(n)}}^{(n)}$ the subtree induced by $\mathbf{\mathfrak{I}_{\ell}}^{(n)}$.

Then $$\mathbb{P}(\mathbf{T}_{\mathbf{\mathfrak{I}_{\ell}}^{(n)}}^{(n)}=\tau) \xrightarrow[n\to\infty]{} \frac{(\ell-1)!}{(2(\ell-1))!}.$$
\end{thm}

\begin{proof}
Since $\mathbf{\mathfrak{I}_{\ell}}^{(n)}$ is independent of $\mathbf{T}^{(n)}$, $$\mathbb{P}(\mathbf{T}_{\mathbf{\mathfrak{I}_{\ell}}^{(n)}}^{(n)}=\tau)=\frac{n![z^n]T_{\tau}}{n(n-1)\dots (n-\ell+1)n![z^n]T}=\frac{[z^n]T_{\tau}}{n(n-1)\dots (n-\ell+1)[z^n]T}$$

By applying the Transfer Theorem \cite[Corollary VI.1 p.392]{flajolet2009analytic} to \cref{ee1}, we get 
$$[z^n]T_{\tau}\sim  \frac{ \kappa}{(1+P^{\bullet}(R)) 2^{2\ell-2}} \frac{n^{\frac{2\ell-3}{2}}}{\Gamma\left(\frac{2\ell-1}{2}\right)R^n}$$

and by \cref{cor1} we obtain
$$n\times \cdots\times (n-\ell+1)[z^n]T\sim n^{\ell} \frac{\kappa}{\sqrt{\pi}(1+\mathcal{P}^{\bullet}(R))} \frac{1}{R^{n}n^{\frac{3}{2}}}.$$

Thus when $n$ goes to infinity \[\mathbb{P}(\mathbf{T}_{\mathbf{\mathfrak{I}_{\ell}}^{(n)}}^{(n)}=\tau)\to \frac{\sqrt{\pi}}{2^{2\ell-2}\Gamma\left(\frac{2\ell-1}{2}\right)}=\frac{(\ell-1)!}{(2(\ell-1))!}\qedhere\]  \end{proof}

Combining \Cref{111} and \Cref{110} proves \Cref{thm6.7} of which \Cref{cgbintro} is a particular case.

\begin{thm}\label{thm6.7}
Let $\mathbf{G}^{(n)}$ be a uniform random graph in $\mathcal{G}_{\mathcal{P},\mathcal{P}^{\bullet}}$ with $n$ vertices. We have the following convergence in distribution in the sense of graphons:

$$W_{\mathbf{G}^{(n)}}\stackrel{n\to \infty}{\longrightarrow} \mathbf{W}^{\frac{1}{2}}$$

\noindent where $\mathbf{W}^{\frac{1}{2}}$ is the Brownian cographon of parameter $\frac{1}{2}$.
\end{thm}

\subsection{Number of induced prime subgraphs}

We now estimate for a prime graph $H$  the number $Occ_H(\mathbf{G}^{(n)})$ of induced occurences of $H$ in $\mathbf{G}^{(n)}$ and show that in average it is null, linear or of order $n^{\frac{3}{2}}$. 

We first observe that substitution trees encoding prime graphs have a very simple structure.

\begin{lem}
Let $H$ be a prime graph. If $t$ is a substitution tree such that $H=\mathrm{Graph}(t)$, $t$ is reduced to a single internal node decorated with a relabeling of $H$ with $|H|$ leaves.
\end{lem}

\begin{proof}
Let $t$ be such a tree and $r$ its root. To every element $t'$ of $t_r$ we can associate a module of H by taking the vertices whose labels are the labels of the leaves of $t'$. Thus $t_r$ is a union of leaves, and the decoration of the root is a relabeling of $H$.
\end{proof}

We say that $H$ verifies $(A)$ if there exists $a\in\{1,\dots, \ell\}$ such that $\mathrm{Occ}_{H,a,\mathcal{P}^{\bullet}}(R)>0$.

\begin{thm}
\label{thmprime}
Let $H$ be a prime graph and let $\ell$ be its size. For $n\geq \ell$, let $\mathbf{G}^{(n)}$ be a uniform random graph in $\mathcal{G}_{\mathcal{P},\mathcal{P}^{\bullet}}$ with $n$ vertices.

Then if $H$ verifies $(A)$, 

$$\hspace{-2,1cm}\mathbb{E}[\mathrm{Occ}_{H}(\mathbf{G}^{(n)})]\sim K_H n^{\frac{3}{2}}\qquad\text{with}\qquad K_H=\frac{R^{\ell-1}\sqrt{\pi}\sum\limits_{a\in\{1,\dots, \ell\}} \mathrm{Occ}_{H,a,\mathcal{P}^{\bullet}}(R)}{\kappa (1+P^{\bullet}(R))}$$

\noindent otherwise,

$$\mathbb{E}[\mathrm{Occ}_{H}(\mathbf{G}^{(n)})]\sim K_H n\qquad\text{with}\qquad K_H=\left(\frac{1-P^{\bullet}(R)}{1+P^{\bullet}(R)}\mathrm{Occ}_{H,\mathcal{P}^{\bullet}}(R)+\mathrm{Occ}_{H,\mathcal{P}}(R)\right) \frac{R^{\ell}}{\kappa^2 }$$

\end{thm}

\begin{proof}
Let $\mathbf{T}^{(n)}$ be a uniform random tree in $\mathcal{T}_{\mathcal{P},\mathcal{P}^{\bullet}}$ with $n$ vertices .

Let $\tau$ be the canonical tree of $H$ and $N_{\mathbf{T}^{(n)},\tau}$ the number of induced subtrees of $\mathbf{T}_n$ isomorphic to $\tau$. Since $\tau$ is the unique substitution tree of $H$, $\mathbb{E}[\mathrm{Occ}_{H}(\mathbf{G}^{(n)})]=\mathbb{E}[N_{\mathbf{T}^{(n)},\tau}]$.

By independence 
$$\mathbb{E}[\mathrm{Occ}_{H}(\mathbf{G}^{(n)})]=\frac{n![z^n]T_{\tau}}{n![z^n]T}=\frac{[z^n]T_{\tau}}{[z^n]T}.$$

From \cref{thmcoeur}, since in this case the only node of $\tau$ is either in $V_0, V_1$ or $V_2$, we have that: $$T_{\tau}=z^{\ell}T^{\mathsf{blo}}\left(T'\left(\sum\limits_{a\in\{1,\dots, \ell\}} \mathrm{Occ}_{H,a,\mathcal{P}^{\bullet}}\right) + T\mathrm{Occ}_{H,\mathcal{P}^{\bullet}}+\mathrm{Occ}_{H,\mathcal{P}}\right).$$

Thus
\begin{itemize}
\item in case $(A)$, with \cref{eqblo,ecut} $$T_{\tau}\sim \frac{R^{\ell}}{R(1+\mathcal{P}^{\bullet}(R))^2}\left(\sum\limits_{a\in\{1,\dots, \ell\}} \mathrm{Occ}_{H,a,\mathcal{P}^{\bullet}}(R)\right)\left(1-\frac{z}{R}\right)^{-1};$$

\item otherwise, $T_{\tau}\sim \left(\frac{1-P^{\bullet}(R)}{1+P^{\bullet}(R)}\mathrm{Occ}_{H,\mathcal{P}^{\bullet}}(R)+\mathrm{Occ}_{H,\mathcal{P}}(R)\right) \frac{R^{\ell}}{\kappa(1+P^{\bullet}(R))}\left(1-\frac{z}{R}\right)^{-\frac{1}{2}}.$
\end{itemize}

By applying the Transfer Theorem \cite[Corollary VI.1 p. 392]{flajolet2009analytic}, 

\begin{itemize}
\item In case $(A)$, $$[z^n]T_{\tau}\sim\frac{R^{\ell}}{R^{n+1}(1+\mathcal{P}^{\bullet}(R))^2}\sum\limits_{a\in\{1,\dots, \ell\}} \mathrm{Occ}_{H,a,\mathcal{P}^{\bullet}}(R)$$

\item Otherwise, $$[z^n]T_{\tau}\sim\left(\frac{1-P^{\bullet}(R)}{1+P^{\bullet}(R)}\mathrm{Occ}_{H,\mathcal{P}^{\bullet}}(R)+\mathrm{Occ}_{H,\mathcal{P}}(R)\right)\frac{R^{\ell}}{\sqrt{\pi}\kappa(1+\mathcal{P}^{\bullet}(R))} \frac{1}{R^nn^{\frac{1}{2}}}$$.
\end{itemize}

By \cref{cor1},
$$[z^n]T\sim \frac{\kappa}{\sqrt{\pi}(1+\mathcal{P}^{\bullet}(R))} \frac{1}{R^nn^{\frac{3}{2}}}.$$

Thus:

\begin{itemize}
    \item In case $(A)$, $$\mathbb{E}[\mathrm{Occ}_{H}(\mathbf{G}^{(n)})]\sim \frac{R^{\ell-1}\sqrt{\pi}\sum\limits_{a\in\{1,\dots, \ell\}} \mathrm{Occ}_{H,a,\mathcal{P}^{\bullet}}(R)}{\kappa (1+P^{\bullet}(R))}n^{\frac{3}{2}},$$
    \item Otherwise, $$\mathbb{E}[\mathrm{Occ}_{H}(\mathbf{G}^{(n)})]\sim \left(\frac{1-P^{\bullet}(R)}{1+P^{\bullet}(R)}\mathrm{Occ}_{H,\mathcal{P}^{\bullet}}(R)+\mathrm{Occ}_{H,\mathcal{P}}(R)\right) \frac{R^{\ell}}{\kappa^2 }n,$$
\end{itemize}
concluding the proof.
\end{proof}

An interesting application of this theorem is the computation of the asymptotic number of $\tilde{P_4}$'s in a random uniform graph of each of the graph classes of \cref{sec:cla}, where $\tilde{P_4}$ is the only labeling of $P_4$ with endpoints $1$ and $4$ and $2$ connected to $1$.  

\begin{lem}
A prime spider has exactly $|K|(|K|-1)$ induced $\tilde{P_4}$. 
A pseudo-spider of size $k$ has exactly $(|K|+2)(|K|-1)$ induced $\tilde{P_4}$.

\end{lem}

\begin{proof}
One can check that for a prime spider, the $P_4's$ are induced by the partial injections $\mathfrak{I}$ whose domain is $\{k,k',f(k),f(k')\}$ for every $(k,k')\in K^2$ with $k\neq k'$ (where $f$ is the function defined in \cref{spider}). For every such domain, only $2$ partial injections are such that the graph induced is $\tilde{P_4}$. Since there is $\binom{|K|}{2}$ possible choices for the domain, we have $|K|(|K|-1)$ induced $\tilde{P_4}$. 

For a pseudo-spider, let $d$ be the duplicate and $d_0$ the original node (as defined in \cref{pseudo}).  The $P_4's$ are induced by the partial injections $\mathfrak{I}$ whose domain is $\{k,k',f(k),f(k')\}$ for every $(k,k')\in K^2$ with $k\neq k'$, and by the partial injections $\mathfrak{I}$ whose domain is $\{d,k',f(d_0),f(k')\}$ (resp.~$\{f^{-1}(d_0),k',d,f(k')\}$) for every $k'\in K$ with $k'\neq d_0$ (resp.~$k'\neq f^{-1}(d_0)$) if $d_0$ is in $K$ (resp.~in $S$). For every such domain, only $2$ partial injections are such that the graph induced is $\tilde{P_4}$. Since there is $\binom{|K|}{2}+K-1$ possible choices for the domain, we have $|K|(|K|-1)+2(|K|-1)=(|K|+2)(|K|-1)$ induced $\tilde{P_4}$.\end{proof}

\begin{rem}
Note that the proof of this lemma implies that $\mathrm{Occ}_{\tilde{P_4},a,\mathcal{P}^{\bullet}}=0$ for all the graph classes mentionned in \cref{sec:cla}.
\end{rem}

\begin{thm}
\label{thmfin}
For each graph class introduced in \Cref{sec:cla}, we have the following expressions for $\mathrm{Occ}_{\tilde{P_4},\mathcal{P}}$ and $\mathrm{Occ}_{\tilde{P_4},\mathcal{P}^{\bullet}}$: \medskip

\begin{center}
\begin{tabular}{|l|l|}
  \hline
  $P_4$-tidy  & $\mathrm{Occ}_{\tilde{P_4},\mathcal{P}_{\mathrm{tidy}}^{\bullet}}(z)=(2+16z+4z^3)\exp(z^2)-1-8z$ \\
   & $\mathrm{Occ}_{\tilde{P_4},\mathcal{P}_{\mathrm{tidy}}}(z)=\mathrm{Occ}_{\tilde{P_4},\mathcal{P}_{\mathrm{tidy}}^{\bullet}}(z)+5z$ \\
   \hline
  $P_4$-lite  & $\mathrm{Occ}_{\tilde{P_4},\mathcal{P}_{\mathrm{lite}}^{\bullet}}(z)=(2+16z+4z^3)\exp(z^2)-1-8z$  \\
   & $\mathrm{Occ}_{\tilde{P_4},\mathcal{P}_{\mathrm{lite}}}(z)=\mathrm{Occ}_{\tilde{P_4},\mathcal{P}_{\mathrm{lite}}^{\bullet}}(z)+4z$ \\
   \hline
  $P_4$-extendible  & $\mathrm{Occ}_{\tilde{P_4},\mathcal{P}_{\mathrm{ext}}^{\bullet}}(z)=1+8z$ \\
  & $\mathrm{Occ}_{\tilde{P_4},\mathcal{P}_{\mathrm{ext}}}=\mathrm{Occ}_{\tilde{P_4},\mathcal{P}_{\mathrm{ext}}^{\bullet}}(z)+5z$ \\
  \hline
  $P_4$-sparse & $\mathrm{Occ}_{\tilde{P_4},\mathcal{P}_{\mathrm{spa}}^{\bullet}}(z)=\mathrm{Occ}_{\tilde{P_4},\mathcal{P}_{\mathrm{spa}}}(z)=2\exp(z^2)-1$ \\
  \hline
  $P_4$-reducible & $\mathrm{Occ}_{\tilde{P_4},\mathcal{P}_{\mathrm{red}}^{\bullet}}(z)=\mathrm{Occ}_{\tilde{P_4},\mathcal{P}_{\mathrm{red}}}=1$  \\
  \hline
  $P_4$-free & $\mathrm{Occ}_{\tilde{P_4},\mathcal{P}_{\mathrm{cog}}^{\bullet}}(z)=\mathrm{Occ}_{\tilde{P_4},\mathcal{P}_{\mathrm{cog}}}(z)=0$  \\
  \hline
\end{tabular}
\end{center}

\end{thm}

\begin{proof}
We only detail the computation of $\mathrm{Occ}_{\tilde{P_4},\mathcal{P}_{\mathrm{tidy}}^{\bullet}}$ and $\mathrm{Occ}_{\tilde{P_4},\mathcal{P}_{\mathrm{tidy}}}$ for $P_4$-tidy graphs as this is the most involved case. Note that, with the notations of \cref{sec:exa},

\begin{align*}
&\mathrm{Occ}_{\tilde{P_4}, \mathcal{P}}(z)=\sum\limits_{n\in\N}\sum\limits_{H\in \mathcal{R}_{\mathcal{P}_n}}\sum\limits_{H'\sim H}\frac{\mathrm{Occ}_{\tilde{P_4}}(H)z^{N(H)-4}}{N(H)!}=\sum\limits_{n\in\N}\sum\limits_{H\in \mathcal{R}_{\mathcal{P}_n}}\frac{\mathrm{Occ}_{\tilde{P_4}}(H)z^{N(H)-4}}{|\mathrm{Aut}(H)|}
\end{align*}
and similarly

\begin{align*}
&\mathrm{Occ}_{\tilde{P_4}, \mathcal{P}^{\bullet}}(z)=\sum\limits_{n\in\N}\sum\limits_{H\in \mathcal{R}_{\mathcal{P}^{\bullet}_n}}\sum\limits_{H'\sim H}\frac{\mathrm{Occ}_{\tilde{P_4}}(H)z^{N(H)-4}}{N(H)!}=\sum\limits_{n\in\N}\sum\limits_{H\in \mathcal{R}_{\mathcal{P}_n}}\frac{\mathrm{Occ}_{\tilde{P_4}}(H)z^{N(H)-4}}{|\mathrm{Aut}(H)|}
\end{align*}

According to \Cref{ttidy}, $\mathcal{P}_{\mathrm{tidy}}$ is composed of one $C_5$ that has $10$ automorphisms and $10$ induced $\tilde{P_4}$ and all its relabelings, one $P_5$, and one $\overline{P_5}$ that both have $2$ automorphisms and $4$ induced $\tilde{P_4}$'s and all their relabelings.

For $k\geq 3$ (resp.~$k=2$), there are thin and fat spiders corresponding to the $2$ (resp.~$1$) different orbits of the action $\Phi_{2k}$ over prime spiders of size $2k$, each having $k!$ automorphisms and $k(k-1)$ $\tilde{P_4}$'s. 

For $k\geq 3$ (resp.~$k=2$), there are thin and fat pseudo-spiders, the duplicated vertex can come from $K$ or $S$, and can be connected or not to the initial vertex. These $8$ (resp.~$4$) cases correspond to the $8$ (resp.~$4$) different orbits of the action $\Phi_{2k+1}$ over pseudo-spiders of size $2k+1$, each having $2(k-1)!$ automorphisms and $(k+2)(k-1)$ $\tilde{P_4}$'s.

Thus we have 

\begin{align*}
\mathrm{Occ}_{\tilde{P_4}, \mathcal{P}_{\mathrm{tidy}}}(z)&=z+\frac{4z}{2}+\frac{4z}{2}+\frac{2}{2}+2\sum\limits_{k\geq 3}\frac{k(k-1)z^{2k-4}}{k!}+4\frac{4z}{2}+8\sum\limits_{k\geq 3}\frac{(k+2)(k-1)z^{2k-3}}{2 (k-1)!}\\
&=5z+1+2\sum\limits_{k\geq 1}\frac{z^{2k}}{k!}+8z+4\sum\limits_{k\geq 1}\frac{(k+4)z^{2k+1}}{k!}\\
&=5z+1+2\exp(z^2)-2+8z+4\sum\limits_{k\geq 0}\frac{z^{2k+3}}{k!}+16\sum\limits_{k\geq 1}\frac{z^{2k+1}}{k!}\\
&=5z+2\exp(z^2)-1+4z^3\exp(z^2)+16z\exp(z^2)-8z\\
&=5z+(2+16z+4z^3)\exp(z^2)-1-8z
\end{align*}

Now let us compute $\mathrm{Occ}_{\tilde{P_4}, \mathcal{P}^{\bullet}_{\mathrm{tidy}}}(z)$. For $k\geq 3$ (resp.~$k=2$), there are thin and fat spiders with blossom corresponding to the $2$ (resp.~$1$) different orbits of the action $\Phi_{2k}$ over blossomed prime spiders $G$ with $2k$ non blossomed vertices, each having $k!$ automorphisms and $k(k-1)$ $\tilde{P_4}$'s. 

For $k\geq 3$ (resp.~$k=2$), there are thin and fat pseudo-spiders, the duplicated vertex can come from $K$ or $S$, and can be connected or not to the initial vertex. These $8$ (resp.~$4$) cases correspond to the $8$ (resp.~$4$) different orbits of the action $\Phi_{2k+1}$ over blossomed pseudo-spiders  with $2k+1$ non blossomed vertices, each having $2(k-1)!$ automorphisms and $(k+2)(k-1)$ $\tilde{P_4}$'s. 

Hence
$$\mathrm{Occ}_{\tilde{P_4}, \mathcal{P}^{\bullet}_{\mathrm{tidy}}}(z)=\frac{2}{2}+2\sum\limits_{k\geq 3}\frac{k(k-1)z^{2k-4}}{k!}+4\frac{4z}{2}+8\sum\limits_{k\geq 3}\frac{(k+2)(k-1)z^{2k-3}}{2 (k-1)!}$$
Thus $\mathrm{Occ}_{\tilde{P_4}, \mathcal{P}^{\bullet}_{\mathrm{tidy}}}(z)+5z=\mathrm{Occ}_{\tilde{P_4}, \mathcal{P}_{\mathrm{tidy}}}(z)$ which gives the announced result.\end{proof}

Combining \Cref{thmprime}, \Cref{thmfin} and the remark above, we get that $\tilde{P_4}$ does not verify $(A)$, thus $\tilde{P_4}$ belongs to the linear case of \Cref{thmprime}:

\begin{cor}\label{cora}
Let $\mathbf{G}^{(n)}$ be a graph of size $n$ taken uniformly at random in one of the following families: $P_4$-sparse, $P_4$-tidy, $P_4$-lite, $P_4$-extendible, $P_4$-reducible or $P_4$-free.
Then $\mathbb{E}[\mathrm{Occ}_{\tilde{P_4}}(\mathbf{G}^{(n)})]\sim K_{\tilde{P_4}} n$ where $K_{\tilde{P_4}}$ is defined in \Cref{thmprime}. 

\end{cor}

Here are the numerical approximations of $K_{\tilde{P_4}}$ in the different cases: \medskip

\begin{center}
\begin{tabular}{|c|c|}
  \hline
  class of graph & $K_{\tilde{P_4}}$  \\
  \hline
  $P_4$-tidy  & $0.29200322$  \\
  $P_4$-lite  & $0.28507010$  \\
  $P_4$-extendible  & $0.24959979$ \\
  $P_4$-sparse & $0.10280703$  \\
  $P_4$-reducible & $0.08249263$  \\
  $P_4$-free & $0$ \\
  \hline
\end{tabular}
\end{center}
\medskip

\textbf{Acknowledgements.} I would like to thank Lucas Gerin and Frédérique Bassino for useful discussions and for carefully reading many earlier versions of this manuscript.

\bibliographystyle{plain}
\bibliography{biblio}

\vfill 
\noindent \textsc{Théo Lenoir} \verb|theo.lenoir@polytechnique.edu|\\
\textsc{Cmap, Cnrs}, \'Ecole polytechnique,\\
Institut Polytechnique de Paris,\\
91120 Palaiseau, France

\end{document}